\documentclass[12pt]{article}

\usepackage{mathrsfs}
\usepackage{amsmath,extpfeil}
\usepackage{stmaryrd}
\usepackage{mathrsfs}
\usepackage{url}
\usepackage{indentfirst}                          
\usepackage{enumerate}
\usepackage{amsmath,amsfonts,amssymb,amsthm}
\usepackage{amsmath,amssymb,amsthm,amscd}
\usepackage{graphicx,mathrsfs}
\usepackage{appendix}
\usepackage[numbers,sort&compress]{natbib}
\usepackage{color}
\usepackage[colorlinks, linkcolor=blue, citecolor=blue]{hyperref}

\usepackage{sidecap}
\usepackage{float}
\usepackage{extarrows}
\usepackage{booktabs}
\usepackage{verbatim}
\usepackage[usenames,dvipsnames]{xcolor}

\newtheorem{theorem}{Theorem}[section]
\newtheorem{lemma}[theorem]{Lemma}

\theoremstyle{definition}
\newtheorem{remark}[theorem]{Remark}

\def\XXint#1#2#3{{\setbox0=\hbox{$#1{#2#3}{\int}$}
         \vcenter{\hbox{$#2#3$}}\kern-.5\wd0}}

\numberwithin{equation}{section}

\allowdisplaybreaks[4]

\begin{document}

\title{Blowup dynamics for mass critical Half-wave equation in 2D }
\author{ Vladimir Georgiev and Yuan  Li}
\date{}
\maketitle

\begin{abstract}
We consider the half-wave equation $iu_t=Du-|u|u$ in two dimensions. For the initial data $u_0(x)\in H^{s}(\mathbb{R}^2)$, $s\in\left(\frac{3}{4},1\right)$, we obtain the non-radial ground state mass blow-up solutions with the blow-up speed $\|D^{\frac{1}{2}}u(t)\|_{L^2}\sim\frac{1}{|t|}$ as $t\to 0^-$. Our result  extends  a recent work of [ Georgiev and Li, Comm. Partial Differential Equations 47 (2022), no. 1, 39–88.], where the radial ground state mass blow up solutions were  constructed.

\end{abstract}

\section{Introduction and Main Result}
\noindent
In this paper, we consider the half-wave (first-order wave) equation in two dimensions
\begin{equation}\label{equ-1-hf-2}
\begin{cases}
i\partial_tu=Du-|u|u,\\
u(t_0,x)=u_0(x),\ u:I\times\mathbb{R}\rightarrow\mathbb{C}.
\end{cases}
\end{equation}
Here, $I\subset\mathbb{R}$ is an interval containing the initial time $t_0\in\mathbb{R}$, and
\begin{align*}
\widehat{(Df)}(\xi)=|\xi|\hat{f}(\xi)
\end{align*}
denotes the first-order nonlocal fractional derivative.  Let us review some basic properties of this equation. The Cauchy problem
\eqref{equ-1-hf-2} is an infinite-dimensional Hamiltonian system, which have the following three conservation laws:
\begin{align*}
\text{Mass:}\ \ M(u)&=\int_{\mathbb{R}^2}|u(t,x)|^2dx=M(u_0),\\
\text{Energy:}\ \ E(u)&=\frac{1}{2}\int_{\mathbb{R}^2}\bar{u}(t,x)Du(t,x)dx
-\frac{1}{3}\int_{\mathbb{R}^2}|u(t,x)|^{3}dx=E(u_0),\\
\text{Momentum:}\ \ P(u)&=\int_{\mathbb{R}^2}-i\nabla u(t,x)\bar{u}(t,x)dx=P(u_0).
\end{align*}
The equation \eqref{equ-1-hf-2} possesses a large group of symmetries: If $u(t,x)$ is a solution to \eqref{equ-1-hf-2}, then for all $(t_0,x_0,\gamma_0,\lambda_0)\in\mathbb{R}\times\mathbb{R}^2\times\mathbb{R}\times\mathbb{R}_*^+$, so is
\[u(t,x)\mapsto \lambda_0u(\lambda_0t-t_0,\lambda_0x-x_0)e^{i\gamma_0}.\]
From the scaling invariant, it is known that equation \eqref{equ-1-hf-2} is mass critical.

This equation has recently gained much attention. For the general power-type nonlinear terms, \cite{BGV2018MN,HW2019SM,D2018DCDS,FGO2019JMPA,KLR2013ARMA} studied the local/global well-posedness. In particular, for $F(u)=|u|^{p}u$, Krieger, Lenzmann and Rapha\"el \cite{KLR2013ARMA} considered the local well-posedness in one dimensional case; Bellazzini, Georgiev and Visciglia \cite{BGV2018MN} obtained the local existence in the space $H^1_{rad}(\mathbb{R}^N)$, $N\geq2$; furthermore, Hidano and Wang \cite{HW2019SM} improved this result and established the local existence in the space $H^s_{rad}(\mathbb{R}^N)$, 
$s\in(\frac{1}{2},1)$ and $N\geq2$, and in $H^s(\mathbb{R}^N)$, $s\in\left(\max\left\{\frac{N-1}{2},\frac{N+1}{4}\right\},p\right)$ and $s\geq s_c$, where $p>\max\left\{s_c,\frac{N-1}{2},\frac{N+1}{4},1\right\}$, $s_c=\frac{N}{2}-\frac{1}{p}$. For the  blow up solution, Krieger, Lenzmann and Rapha\"el \cite{KLR2013ARMA}  first constructed a minimal mass blow-up solutions to the mass critical Half-wave equation in one dimension, and then the authors \cite{GL2021JFA,GL2022CPDE} obtained the radial ground state mass blow-up solution in the two and three dimension cases, for other blow-up results, one can see \cite{FGO2018DPDE,I2016PAMS,L2022}. In additional, Dinh \cite{D2018DCDS}  established the ill well-posedness and low regularity data.  Now in the present paper, we aim to extend the result in \cite{GL2022CPDE} and study the non-radial case. 

Now we state our main result.
\begin{theorem}\label{Thm1}
For all $(E_0,P_0)\in \mathbb{R}_+^*\times\mathbb{R}^2$, there exist $t^*<0$ independent of $E_0$, $P_0$, and a ground state mass solution $u\in C^0([t^*,0);H^{s}(\mathbb{R}^2))$, $s\in\left(\frac{3}{4},1\right)$, of equation \eqref{equ-1-hf-2} with
\begin{align}\notag
\|u(t)\|_2=\|Q\|_2,\ E(u(t))=E_0,\ P(u(t))=P_0,
\end{align}
which blows up at time $T=0$. More precisely, it holds that
\begin{align*}
u(t,x)-\frac{1}{\lambda(t)}Q\left(\frac{x-\alpha(t)}{\lambda(t)}\right)
e^{i\gamma(t)}\rightarrow0\ \text{in}\ L^2(\mathbb{R}^2)\ \text{as}\ t\rightarrow0^-,
\end{align*}
where
\begin{align}\notag
\lambda(t)=\lambda^*t^2+\mathcal{O}(t^5),\ \alpha(t)=\mathcal{O}(t^3),\ \gamma(t)=\frac{1}{\lambda^*|t|}+\mathcal{O}(t),
\end{align}
with some constant $\lambda^*>0$, and the blowup speed is given by:
\begin{align*}
\|D^{\frac{1}{2}}u(t)\|_2\sim\frac{C(u_0)}{|t|}\ \text{as}\ t\rightarrow0^{-},
\end{align*}
where $C(u_0)>0$ is constant  depending only  on the initial data $u_0$.
\end{theorem}
Let us make some  comments on the proof of Theorem \ref{Thm1}.

1. We aim to construct an exact solution of the form
\begin{align*}
u(t,x)=\frac{1}{\lambda(t)}[Q_{\mathcal{P}}+\epsilon]
\left(t,\frac{x-\alpha(t)}{\lambda(t)}\right)e^{i\gamma(t)}=\tilde{Q}+\tilde{\epsilon}.
\end{align*}
In case of higher dimensions $ N \geq 2$ the mass critical exponent $1+\frac{2}{N}$ becomes smaller, the modulation estimates, the refine energy estimates and the bootstrap argument (especially the $\|\cdot\|_{H^{1/2+\cdot}}$-norm) become more complicated. In \cite{GL2022CPDE,GL2021JFA}, radial properties play a crucial role in the discussion. Now in the present paper, radial properties are failure and so we need some new techniques to overcome this difficulty.

2. One can control the $\|\cdot\|_{H^s}$ norm  by the weighted $L_t^qH^1_x$-norm in the two dimensional and radial case and this  is  a crucial tool to evaluate some terms in the  refine energy estimates and bootstrap argument. However, in the non-radial case, such a control is impossible, so we need new tools. 

3. Our proof of the existence of ground state blow-up solution in $H^{s}(\mathbb{R}^2)$, $s\in \left(\frac{3}{4},1\right)$ bases on the $L^q_tL^\infty_x$ and $L^\infty_t H^s_x$  Strichartz-type estimates (see Klainerman and Machedon \cite{KM1998IMRN} and Hidano and Wang \cite{HW2019SM}).

This paper is organized as follows: in Section 2, We collect some useful estimates and inequalities; in Section 3, we construct the higher order approximate $Q_{\mathcal{P}}$ solution of the renormalized equation; in Section 4, we decompose the solution and show the estimate the modulation parameters;   energy/virial type estimate and bootstrap argument that will be needed in the construction of the ground state mass blowup solutions; in Section 5, we prove the Theorem \ref{Thm1}; and the finally section is Appendix.

\textbf{Notations}\\
- $(f,g)=\int \bar{f}g$ as the inner product on $L^2(\mathbb{R}^2)$.\\
- $\|\cdot\|_{L^p}$ denotes the $L^p(\mathbb{R}^2)$ norm for $p\geq 1$.\\
- $\widehat{f}$ denotes the Fourier transform of function $f$.\\
- We shall use $X\lesssim Y$ to denote that $X\leq CY$ holds, where the constant $C>0$ may change from line to line, but $C$ is allowed to depend on universally fixed quantities only.\\
- Likewise, we use $X\sim Y$ to denote that both $X\lesssim Y$ and $Y\lesssim X$ hold.

For a sufficiently regular function $f:\mathbb{R}^2\rightarrow\mathbb{C}$, we define the generator of $L^2$ scaling given by
\begin{align}\notag
\Lambda f:=f+x\cdot\nabla f.
\end{align}
Note that the operator $\Lambda$ is skew-adjoint on $L^2(\mathbb{R}^2)$, that is, we have
\begin{align}\notag
(\Lambda f, g)=-(f,\Lambda g).
\end{align}
We write $\Lambda^kf$, with $k\in\mathbb{N}$, for the iterates of $\Lambda$ with the convention that $\Lambda^0f\equiv f$.

In some parts of this paper, it will be convenient to identity any complex-valued function $f:\mathbb{R}^2\rightarrow\mathbb{C}$ with the function $\mathbf{f}:\mathbb{R}^2\rightarrow\mathbb{R}^2$ by setting
\begin{align}\notag
\mathbf{f}={\left[ \begin{array}{c}
f_1\\
f_2
\end{array}
\right ]}={\left[ \begin{array}{c}
\Re f\\
\Im f
\end{array}
\right ]}.
\end{align}
We also define
$$\mathbf{f}\cdot\mathbf{g}=f_1g_1+f_2g_2.$$
Corresponding, we will identity the multiplication by $i$ in $\mathbb{C}$ with the multiplication by the real $2\times 2$-matrix defined as
\begin{align}\notag
J={\left[\begin{array}{cc}
0 & -1\\
1 & 0
\end{array}\right]}.
\end{align}

\section{Preliminaries}
\noindent
In this section, for half-wave equation
\begin{align}\label{equ:2}
    i\partial_tu-Du=G(u),~~u(0,x)=u_0(x),~~(t,x)\in\mathbb{R}\times\mathbb{R}^2.
\end{align}
By Duhamel formula, $u$ is a weak solution of \eqref{equ:2} is equivalent to 
\begin{align*}
    u(t)=U(t)u_0-i\int_0^tU(t-s)G(u(s))ds,
\end{align*}
 where $U(t)=e^{-itD}$. From  \cite{KM1998IMRN,HW2019SM}, we have the following  Strichartz type estimates.
\begin{lemma}\label{lemma:stri} (Strichartz estimates)
Let $q\in(4,\infty)$, $\sigma=1-\frac{1}{q}$. Then  we have the following inequality
\[\|U(t)f\|_{L^qL^\infty}\lesssim \|f\|_{\dot{H}^{\sigma}}.\]
In particular,
\[\|u\|_{\dot{H}^{s}}\lesssim\|u_0\|_{\dot{H}^{s}}+\int_0^t\|G(u)\|_{\dot{H}^{s}},~~s\in[0,1)\]
and
\[\|u\|_{L^qL^\infty}\lesssim\|U(t)u_0\|_{L^qL^\infty}+\|G(u)\|_{L^1\dot{H}^{\sigma}}.\]
\end{lemma}
The following result (generalized Leibniz rule) is proved in \cite{GK1996AJM} (also see \cite{KP1988CPAM,L2007MPAG}) for Riesz and
Bessel potentials of order $s\in\mathbb{R}$.
\begin{lemma}\label{lemma:flr}
(Fractional Leibniz rule) Suppose that $1<p<\infty$, $s\geq 0$, $\alpha\geq0$, $\beta\geq0$ and $\frac{1}{p_i}+\frac{1}{q_i}=\frac{1}{p}$ with $i=1,2$, $1<q_i\leq\infty$, $1<p_i\leq\infty$. Then 
\[\|D^s(fg)\|_{p}\lesssim\|D^{s+\alpha}f\|_{p_1}\|D^{-\alpha}g\|_{q_1}+\|D^{-\beta}f\|_{p_2}\|D^{s+\beta}g\|_{q_2},\]
where
\[D^s=(\mu^2-\Delta)^{\frac{s}{2}},~~\mu\geq0.\]
\end{lemma}
\begin{lemma}\label{lemma:fcr}
(Fractional Chain rule \cite{GOV1994poincare}) Let $F(u)=|u|^{p-1}u$ or $F(u)=|u|^p$, $p>1$. Then we have
\begin{align*}
    &\|F(u)\|_{\dot{H}^s}\lesssim\|u\|_{L^\infty}^{p-1}\|u\|_{\dot{H}^s},~~s\in(0,\min\{p,N/2\}),\\
    &\|F(u)\|_{H^s}\lesssim\|u\|_{L^\infty}^{p-1}\|u\|_{H^s},~~s\in(0,p).
\end{align*}
\end{lemma}

\section{Approximate Blowup Profile}
In this section, we aim to  construct the approximate blowup profile.
We start with a general observation: If $u=u(t,x)$ solves \eqref{equ-1-hf-2}, then we define the function $v=v(s,y)$ by setting
\begin{align*}
u(t,x)=\frac{1}{\lambda(t)}v\left(s,\frac{x-\alpha(t)}{\lambda(t)}\right)
e^{i\gamma(t)},\ \ \frac{ds}{dt}=\frac{1}{\lambda(t)}.
\end{align*}
It is easy to check that $v=v(s,y)$ with $y=\frac{x-\alpha}{\lambda}$ satisfies
\begin{equation*}
i\partial_sv-Dv-v+|v|v=i\frac{\lambda_s}{\lambda}\Lambda v+i\frac{\alpha_s}{\lambda}\cdot\nabla v+\tilde{\gamma_s}v,
\end{equation*}
where we set $\tilde{\gamma_s}=\gamma_s-1$. Here the operators $D$ and $\nabla$ are understood as $D=D_y$ and $\nabla=\nabla_y$, respectively. Following the slow modulated ansatz strategy developed in \cite{RS2011JAMS,KLR2013ARMA} (also see \cite{KMR2009CPAM,GL2021JFA,GL2022CPDE,L2022}), we freeze the modulation
\begin{align}\label{mod1}
-\frac{\lambda_s}{\lambda}=a,\ \ \frac{\alpha_s}{\lambda}=b.
\end{align}
And we look for an approximate solution of the form
\begin{align}\label{eq.der1}
v(s,y) = Q_{\mathcal{P}(s)}(y),\ \mathcal{P}(s)=(a(s),b(s)),
\end{align}
where
\begin{align}\notag
Q_{\mathcal{P}}(y)=Q(y)+ \left(\sum_{k \geq 1}a^{k}R_{k,0}(y) \right) + \left( \sum_{k+\ell\geq 1,\ell\neq0} a^{k}\sum_{j=1}^2b^{\ell}_j R_{k,\ell,j}(y)\right),
\end{align}
where $\mathcal{P} = (a,b) \in \mathbb{R} \times \mathbb{R}^2.$ And the terms $R_{k,0}(y)$, $R_{k,\ell,j}(y)$ are decomposed in real and imaginary parts as follows
\begin{align}\notag
R_{k,0}(y)=T_{k,0}(y)+iS_{k,0}(y),\,R_{k,\ell,j}(y)=T_{k,\ell,j}(y)+iS_{k,\ell,j}(y).
\end{align}
We also use the notation
\begin{align}\notag
T_{k,\ell}=(T_{k,\ell,1},T_{k,\ell,2})~\,\text{and}~\,S_{k,\ell}=(S_{k,\ell,1},S_{k,\ell,2}),~~\text{where}~~\ell\neq0.
\end{align}
We shall define ODE for
$a(s), b(s) $ of type
$$ a_s = P_1(a,b),~~ b_s = P_2(a,b),$$
where $P_1,P_2$ are appropriate polynomials in $a,b$.

Using the heuristic asymptotic expansions 
\[\lambda (t) \sim t^2,~~ |\alpha(t)| \sim t^3,\]
from $\frac{ds}{dt}=\frac{1}{\lambda(t)}$ we see that $s= s_0-1/t$ goes to $\infty$ as $t \nearrow 0$ and $t=1/(s_0-s) \sim -1/s$ as $s \to + \infty.$  Moreover, the modulation relations \eqref{mod1} show that
\[a(s) = -\frac{\lambda_s}{\lambda} \sim \frac{1}{s} , \ |b(s)| = \frac{|\alpha_s|}{\lambda} \sim \frac{1}{s^2}.\]
These asymptotic expansions suggests to define $a(s), b(s)$ so that
\[
a_s=-\frac{a^2}{2},\ b_s=-ab.
\]
Moreover the asymptotic expansions for $a(s),b(s)$ show that we can consider $\mathcal{P} =(a,b)$ close to the origin with norm
\[\|\mathcal{P}\|^2 \sim a^2 + |b|.\]
We adjust the modulation equation for $(a(s),b(s))$ to ensure the solvability of the obtained system, and a specific algebra leads to the laws to leading order:
\[
a_s=-\frac{a^2}{2},\ b_s=-ab.
\]
From \eqref{eq.der1}, we have
$$ \partial_s v = -\frac{a^2}{2}\partial_aQ_{\mathcal{P}}-a\sum_{j=1}^{2}b_j\partial_{b_j}Q_{\mathcal{P}}.$$
Therefore, our aim is  to construct a higher order approximate profile $Q(y,a,b)=Q_{\mathcal{P}}(y)$ that is the approximate solution to
\begin{align}\notag
-i\frac{a^2}{2}\partial_aQ_{\mathcal{P}}-ia\sum_{j=1}^{2}b_j\partial_{b_j}Q_{\mathcal{P}}-DQ_{\mathcal{P}}
-Q_{\mathcal{P}}+ia\Lambda Q_{\mathcal{P}}-i\sum_{j=1}^2b_j\partial_jQ_{\mathcal{P}}+
|Q_{\mathcal{P}}|Q_{\mathcal{P}}=-\Phi_{\mathcal{P}},
\end{align}
where $\mathcal{P}=(a,b)$ is close to $0$ and $\Phi_{\mathcal{P}}$ is some small terms of order $\mathcal{O}(\|P\|^5)=\mathcal{O}(a^5+|b|^{5/2})$.

We have the following result about an approximate blowup profile $\mathbf{Q}_{\mathcal{P}}$, parameterized by $\mathcal{P}=(a,b)$, around the ground state $\mathbf{Q}=[Q,0]^{\top}$.
\begin{lemma}(Approximate Blowup Profile)\label{lemma-3app}
Let $\mathcal{P}=(a,b)$. There exists a smooth function $\mathbf{Q}_{\mathcal{P}}=\mathbf{Q}_{\mathcal{P}}(x)$ of the form
\begin{align*}
\mathbf{Q}_{\mathcal{P}}=&\mathbf{Q}+a\mathbf{R}_{1,0}+\sum_{j=1}^2b_j\mathbf{R}_{0,1,j}+a\sum_{j=1}^2b_j\mathbf{R}_{1,1,j}
+a^2\mathbf{R}_{2,0}\notag\\
&+\sum_{j=1}^2b_j^2\mathbf{R}_{0,2,j}+a^3\mathbf{R}_{3,0}+a^2\sum_{j=1}^2b_j\mathbf{R}_{2,1,j}+a^4\mathbf{R}_{4,0}
\end{align*}
that satisfies the equation
\begin{equation}\label{equ-3approxiamte}
-J\frac{1}{2}a^2\partial_a\mathbf{Q}_{\mathcal{P}}-Ja\sum_{j=1}^{2}b_j\partial_{b_j}\mathbf{Q}_{\mathcal{P}}
-D\mathbf{Q}_{\mathcal{P}}-\mathbf{Q}_{\mathcal{P}}+Ja\Lambda\mathbf{Q}_{\mathcal{P}}-J\sum_{j=1}^{2}b_j\partial_{b_j}\mathbf{Q}_{\mathcal{P}}
+|\mathbf{Q}_{\mathcal{P}}|\mathbf{Q}_{\mathcal{P}}
=-\mathbf{\Phi}_{\mathcal{P}}.
\end{equation}
Here the functions $\{\mathbf{R}_{k,l}\}_{0\leq k\leq3,0\leq l\leq 1}$ with $\mathbf{R}_{k,l}= ( \mathbf{R}_{k,l,1}, \mathbf{R}_{k,l,2})^T$ satisfy the following regularity and decay bounds:
\begin{align*}
\|\mathbf{R}_{k,l}\|_{H^m}+\|\Lambda\mathbf{R}_{k,l}\|_{H^m}
+\|\Lambda^2\mathbf{R}_{k,l}\|_{H^m}\lesssim1,\ &\text{for}\ m\in\{0,1\},\\
|\mathbf{R}_{k,l}|+|\Lambda\mathbf{R}_{k,l}|+|\Lambda^2\mathbf{R}_{k,l}|\lesssim
\langle x\rangle^{-3},\ &\text{for}\ x\in\mathbb{R}^2.
\end{align*}
Moreover, the term on the right-hand side of \eqref{equ-3approxiamte} satisfies
\begin{align*}
\|\mathbf{\Phi}_{\mathcal{P}}\|_{H^m}\lesssim\mathcal{O}(a^5+|b|^2|\mathcal{P}|),\ |\nabla\mathbf{\Phi}_{\mathcal{P}}|\lesssim\mathcal{O}(a^5+|b|^2|\mathcal{P}|)\langle x\rangle^{-3},
\end{align*}
for $m\in\{0,1\}$ and $x\in\mathbb{R}^2$.
\end{lemma}
\begin{proof}
We recall that the definition of the linear operator
\begin{equation}\notag
L={\left[ \begin{array}{cc}
L_+ & 0\\
0 & L_{-}
\end{array}
\right ]}
\end{equation}
acting on $L^2(\mathbb{R}^2,\mathbb{R}^{2})$, where $L_+$ and $L_-$ denote the unbounded operators acting on $L^2(\mathbb{R}^2,\ \mathbb{R}^{2})$ given by
\begin{align}\notag
L_+=D+1-2Q,\ L_-=D+1-Q.
\end{align}
From \cite{FLS2016CPAM}, we have the key property that the kernel of $L$ is given by
\begin{equation}\notag
\ker L=span\left\{{\left[ \begin{array}{c}
\nabla Q\\ 0
\end{array}
\right ]},{\left[ \begin{array}{c}
 0\\Q
\end{array}
\right ]}\right\}.
\end{equation}
Note that the bounded inverse $L^{-1}=diag\{L_+^{-1},L_-^{-1}\}$ exists on the orthogonal complement $\{\ker L\}^{-1}=\{\nabla Q\}^{\bot}\bigoplus\{Q\}^{\bot}$.

$\mathbf{Step~1}$ Determining the functions $\mathbf{R}_{k,l}$.
We discuss our ansatz for $\mathbf{Q}_{\mathcal{P}}$ to solve \eqref{equ-3approxiamte} order by order. The proof of the regularity and decay bounds for the functions  $\mathbf{R}_{k,l}$ will be given further below.

By the same argument as \cite[Lemma 3.1]{GL2022CPDE}, the Order $\mathcal{O}(1)$, $\mathcal{O}(a)$, $\mathcal{O}(a^2)$, $\mathcal{O}(a^3)$ and $\mathcal{O}(a^4)$ can be obtained. Here we omit it. We only give the following identities, which  will be very useful  in the later discussion.
\begin{align*}
    L_-S_{1,0}=\Lambda Q,~~L_+T_{2,0}=\frac{1}{2}S_{1,0}+\Lambda S_{1,0}-\frac{1}{2}S_{1,0}^2, ~~(S_{1,0},S_{1,0})=-2(T_{2,0},Q).
\end{align*}
Those identities can be found in \cite[Lemma 3.1]{GL2022CPDE}.
Here we only give the following that contains the parameter $b$.

$\mathbf{Order}$ $\mathcal{O}(b)$:  Here we need to solve
\begin{align}\notag
L\mathbf{R}_{0,1}=-J\nabla\mathbf{Q}.
\end{align}
We observe the orthogonality $J\nabla\mathbf{Q}=[0,\nabla Q]^{\top}\bot\ker L$, since $(\nabla Q,Q)=0$ holds. Thus there is a unique solution $\mathbf{R}_{0,1}\perp\ker L$, which we denote as
\begin{align}\notag
\mathbf{R}_{0,1}=-L^{-1}J\nabla\mathbf{Q}=
{\left[\begin{array}{c}
0 \\ -L_{-}^{-1}\nabla Q
\end{array}\right]}.
\end{align}

$\mathbf{Order}$ $\mathcal{O}(ab)$:
We find that $\mathbf{R}_{1,1}$ has to solve the equation
\begin{align}\label{3app-1}
L\mathbf{R}_{1,1}=-J\mathbf{R}_{0,1}+J\Lambda\mathbf{R}_{0,1}-J\nabla\mathbf{R}_{1,0}
+(\mathbf{R}_{1,0}\cdot\mathbf{{R}}_{0,1})Q^{-1}\mathbf{Q},
\end{align}
where $\mathbf{Q}=[Q,0]^{\top}$ and we use the fact that $\mathbf{Q}\cdot\mathbf{R}_{1,0}=\mathbf{Q}\cdot\mathbf{R}_{0,1}=0$. Now, we need to prove
\begin{align}\label{3app-3}
\text{the right-hand side of the above \eqref{3app-1} is}\ \bot \ker L.
\end{align}
Indeed, we note that
\begin{align*}
\mathbf{R}_{1,0}={\left[\begin{array}{c}
0\\S_{1,0}
\end{array}
\right]},~\ &\text{with}~\ L_{-}S_{1,0}=\Lambda Q,\\
\mathbf{R}_{0,1}={\left[\begin{array}{c}
0\\S_{0,1}
\end{array}
\right]},~\ &\text{with}~\ L_{-}S_{0,1}=-\nabla Q.
\end{align*}
Notice that $S_{0,1}=(S_{0,1,1},S_{0,1,2})$ is a vector. Therefore the orthogonality condition \eqref{3app-3} is equivalent to
\begin{align}\label{3app-2}
(\nabla Q,S_{0,1})-(\nabla Q,\Lambda S_{0,1})+(\nabla Q,\nabla S_{1,0})+(\nabla Q,S_{1,0}S_{0,1})=0.
\end{align}
To see that this holds true, we argue as follows. Using the  commutator formula $[\Lambda,\nabla]=-\nabla$ and intergrating by part, we obtain
\begin{align*}
-(\nabla Q,\Lambda S_{0,1})&=(\Lambda\nabla Q,S_{0,1})=(\nabla\Lambda Q,S_{0,1})-(\nabla Q,S_{0,1})\notag\\
&=(\nabla L_{-}S_{1,0},S_{0,1})-(\nabla Q,S_{0,1}).
\end{align*}
Next, since $L_{-}$ is self-adjoint and the definition of $S_{0,1}$, for any function $F$, we have
\begin{align*}
(\nabla L_{-}F,S_{0,1})+(\nabla Q, \nabla F)&=-(L_{-}F,\nabla S_{0,1})-(L_{-}S_{0,1},\nabla F)=(F,[\nabla,L_{-}]S_{0,1})\\
&=-(F,\nabla QS_{0,1}),
\end{align*}
where we use the commutator formulate $[\nabla,L_{-}]=-\nabla Q $. By combining the above equalities, we conclude that \eqref{3app-2} holds. This means that the \eqref{3app-3} holds, and hence there is a unique solution $\mathbf{R}_{1,1}\bot\ker L$ of the equation \eqref{3app-1}. Moreover, we note that
\begin{align}\notag
\mathbf{R}_{1,1}={\left[\begin{array}{c}T_{1,1}\\0\end{array}\right]}.
\end{align}

$\mathbf{Order}$ $\mathcal{O}(b^2)$, where $b^2=(b_1^2,b_2^2)$.

We find the equation
\begin{align}\notag
L\mathbf{R}_{0,2}=-J\nabla  \mathbf{R}_{0,1}+\frac{1}{2}|\mathbf{R}_{0,1}|^2Q^{-1}\mathbf{Q}.
\end{align}
Since $\mathbf{R}_{0,1}=[0,S_{0,1}]^{\top}$ with $L_{-}S_{0,1}=-\nabla Q$ and $\mathbf{Q}=[Q,0]^{\top}$, the solvability condition reads
\begin{align}\notag
(\nabla Q,\nabla S_{0,1})+\frac{1}{2}(\nabla Q,S_{0,1}^2)=0.
\end{align}
Obviously, this is true, since $Q$ is radial function and $S_{0,1,j}$ is antisymmetry function. Hence there exists a unique solution $\mathbf{R}_{0,2,j}\bot\ker L$, and we have
\begin{align}\notag
\mathbf{R}_{0,2,j}={\left[\begin{array}{c}L_{+}^{-1}\left(\partial_{y_j} S_{0,1,j}+\frac{1}{2}|S_{0,1,j}|^2\right)\\0\end{array}\right]},~~j=1,2.
\end{align}

$\mathbf{Order}$ $\mathcal{O}(a^2b_j),\,j=1,2$:
Note that $\mathbf{R}_{1,0}\cdot\mathbf{Q}=\mathbf{R}_{1,1,j}\cdot\mathbf{R}_{1,0}
=\mathbf{R}_{0,1,j}\cdot\mathbf{R}_{2,0}=0$. We find the equation
\begin{align}\label{3app-7}
L\mathbf{R}_{2,1,j}=&\frac{3}{2}J\mathbf{R}_{1,1,j}+J\Lambda\mathbf{R}_{1,1,j}
-J\partial_j\mathbf{R}_{2,0}+\Re\mathbf{R}_{1,1,j}\mathbf{R}_{1,0}\notag\\
&+(\mathbf{R}_{1,0}\cdot\mathbf{\bar{R}}_{0,1,j})\mathbf{R}_{1,0}
+\frac{|\mathbf{R}_{1,0}|^2\mathbf{R}_{0,1,j}}{Q}.
\end{align}
Using the symmetries of the previously constructed functions, we can check that
\begin{align}\notag
\text{Right-hand side of \eqref{3app-7}} \bot\ker L,
\end{align}
Since $(g,Q)=0$ for any antisymmetry function $g\in L^{2}(\mathbb{R}^2)$. Thus there is a unique solution $\mathbf{R}_{2,1,j}\bot\ker L$ of the equation \eqref{3app-7}, and we set  that $\mathbf{R}_{2,1,j}=[0,S_{2,1,j}]^{\top}$ with some radial function $S_{2,1,j}$, $j=1,2$.

$\mathbf{Step~2}:$ By the similar argument as \cite{GL2022CPDE,GL2021JFA}, the regularity and decay bounds can be easily obtained.  
The pointwise estimates for the terms $R_{j,k}$ follow easily by the argument presented in \cite[Appendix A]{GL2022CPDE}.

 Now the proof of Lemma \ref{lemma-3app} is now complete.
\end{proof}

\begin{remark}\label{remark:1}
. Note that $L_{-}>0$ on $Q^{\bot}$ and we have $S_{1,0}\bot Q$ and $S_{0,1,j}\bot Q$, $j=1,2$.

2. The proof of Lemma \ref{lemma-3app} will actually show that the functions $\{\mathbf{R}_{k,l}\}$ have the following symmetry structure (symmetry terms are even function, while antisymmetry stands for odd functions)
\begin{align*}
&\mathbf{R}_{1,0}=\left[\begin{array}{c}0\\symmetry\end{array}\right],\
\mathbf{R}_{0,1,j}=\left[\begin{array}{c}0\\antisymmetry\end{array}\right],\
\mathbf{R}_{1,1,j}=\left[\begin{array}{c}antisymmetry\\0\end{array}\right],\\
&\mathbf{R}_{2,0}=\left[\begin{array}{c}symmetry\\0\end{array}\right],\
\mathbf{R}_{0,2,j}=\left[\begin{array}{c}symmetry\\0\end{array}\right],\
\mathbf{R}_{3,0}=\left[\begin{array}{c}0\\symmetry\end{array}\right],\\
&\mathbf{R}_{2,1,j}=\left[\begin{array}{c}0\\antisymmetry\end{array}\right],\
\mathbf{R}_{4,0}=\left[\begin{array}{c}symmetry\\0\end{array}\right].
\end{align*}
These symmetry properties will be of essential use in the sequel.
\end{remark}

We now turn to some key properties of the approximate blowup profile $\mathbf{Q}_{\mathcal{P}}$ constructed in Lemma \ref{lemma-3app}.
\begin{lemma}\label{lemma-3app-2}
The mass, the energy and the linear momentum of $\mathbf{Q}_{\mathcal{P}}$ satisfy
\begin{align*}
\int|\mathbf{Q}_{\mathcal{P}}|^2&=\int  Q^2+\mathcal{O}(a^4+|b|^2+|b||\mathcal{P}|^2);\\
E(\mathbf{Q}_{\mathcal{P}})&=e_1a^2+\mathcal{O}(a^4+|b|^2+|b||\mathcal{P}|^2);\\
P(\mathbf{Q}_{\mathcal{P}})&=p_1 b+\mathcal{O}(a^4+|b|^2+|b||\mathcal{P}|^2).
\end{align*}
Here $e_1>0$ and $p_1>0$ are the positive constants given by
\begin{align}\notag
e_1=\frac{1}{2}(L_{-}S_{1,0},S_{1,0}),\ p_1=2\int_{\mathbb{R}^2}L_{-}S_{0,1}\cdot S_{0,1},
\end{align}
where $S_{1,0}$ and $S_{0,1}$ satisfy $L_{-}S_{1,0}=\Lambda Q$ and $L_{-}S_{0,1}=-\nabla Q$, respectively.
\end{lemma}
\begin{proof}
The mass and energy estimates is similar to  \cite[Lemma 3.2]{GL2022CPDE}. A small modification due to the parameter $b$.

For the expansion of the linear momentum functional, we observe that $P(\mathbf{f})=2\int f_1\nabla f_2$, where $\mathbf{f}=[f_1,f_2]^{\top}$. Hence
\begin{align*}
P(\mathbf{Q}_{\mathcal{P}})=&2a\int Q\nabla S_{1,0}+2\sum_{j=1}^2b_j\int Q\nabla S_{0,1,j}+2a^2\sum_{j=1}^2b_j\int T_{1,1,j}\nabla S_{1,0}+2a^3\int T_{2,0}\nabla S_{1,0}\\
&+\mathcal{O}(a^4+|b|^2+|b||\mathcal{P}|^2)\\
=&-2\sum_{j=1}^2b_j( \nabla Q, S_{0,1,j})+\mathcal{O}(a^4+|b|^2+|b||\mathcal{P}|^2)\\
=&2\sum_{j=1}^2b_j(L_{-}S_{0,1},S_{0,1,j})+\mathcal{O}(a^4+|b|^2+|b||\mathcal{P}|^2)\\
=&bp_1+\mathcal{O}(a^4+|b|^2+|b||\mathcal{P}|^2),
\end{align*}
Since $L_{-}S_{0,1}=-\nabla Q$, and using that $\int Q\nabla S_{1,0}+\int T_{1,1,j}\nabla S_{1,0}+\int T_{2,0}\nabla S_{1,0}=0$ due to the fact that $Q,S_{1,0},T_{2,0}$ are the radial symmetry functions.

The proof of Lemma \ref{lemma-3app-2} is now complete.
\end{proof}

\section{Modulation Estimates and Energy Estimates}\label{section-mod-estimate}
\subsection{Geometrical Decomposition and Modulation Equations}
Let $u\in H^{s}(\mathbb{R}^2)$, $s\in\left(\frac{3}{4},1\right)$ be a solution of \eqref{equ-1-hf-2} on some time interval $[t_0,t_1]$ with $t_1<0$. Assume that $u(t)$ admits a geometrical decomposition of the form
\begin{align}\label{mod-decomposition}
u(t,x)=\frac{1}{\lambda(t)}[Q_{\mathcal{P}(t)}+\epsilon]
\left(s,\frac{x-\alpha(t)}{\lambda(t)}\right)
e^{i\gamma(t)},\ \ \frac{ds}{dt}=\frac{1}{\lambda(t)},
\end{align}
with $\mathcal{P}(t)=(a(t),b(t))$, and we impose the uniform smallness bound
\begin{align*}
a^2(t)+|b(t)|+\|\epsilon\|_{H^s}^2\ll1.
\end{align*}
Furthermore, we assume that $u(t)$ has almost critical mass in the sense that
\begin{align*}
\left|\int|u(t)|^2-\int Q^2\right|\lesssim\lambda^2(t),\ \ \forall t\in[t_0,t_1].
\end{align*}
To fix the modulation parameters $\{a(t),b(t),\lambda(t),\alpha(t),\gamma(t)\}$ uniquely, we impose the following orthogonality conditions on $\epsilon=\epsilon_1+i\epsilon_2$ as follows:
\begin{equation}\label{mod-orthogonality-condition}
\begin{aligned}
(\epsilon_2,\Lambda Q_{1\mathcal{P}})-(\epsilon_1,\Lambda Q_{2\mathcal{P}})=0,\\
(\epsilon_2,\partial_aQ_{1\mathcal{P}})-(\epsilon_1,\partial_aQ_{2\mathcal{P}})=0,\\
(\epsilon_2,\partial_{b_j}Q_{1\mathcal{P}})-(\epsilon_1,\partial_{b_j}Q_{2\mathcal{P}})=0,\\
(\epsilon_2,\partial_{y_j} Q_{1\mathcal{P}})-(\epsilon_1,\partial_{y_j} Q_{2\mathcal{P}})=0,\\
(\epsilon_2,\rho_1)-(\epsilon_1,\rho_2)=0,
\end{aligned}
\end{equation}
where $j=1,2$, $Q_{\mathcal{P}}=Q_{1\mathcal{P}}+iQ_{2\mathcal{P}}$,which (in terms of the vector notation used in Section 3) means that
\begin{align}\notag
\mathbf{Q}_{\mathcal{P}}={\left[\begin{array}{c}
Q_{1\mathcal{P}}\\Q_{2\mathcal{P}}
\end{array}
\right]},
\end{align} 
and 
 the function $\rho=\rho_1+i\rho_2$ is defined by
\begin{align}\label{mod-definition-rho}
L_{+}\rho_1=S_{1,0}, L_{-}\rho_2=aS_{1,0}\rho_1+a\Lambda\rho_1-2aT_{2,0}
+b\cdot S_{0,1}\rho_1-b\cdot\nabla\rho_1-b\cdot T_{1,1},
\end{align}
where $S_{1,0}$, $T_{2,0}$ and $T_{1,1,j}$ are the functions introduced in the proof of Lemma \ref{lemma-3app}. Note that $L_{+}^{-1}$ exists on $L^2(\mathbb{R}^2)$ since $S_{1,0}\perp \nabla Q$ (see Remark \ref{remark:1}) and thus $\rho_1$ is well-defined. Moreover, it is easy to see that the right-hand side in the equation for $\rho_2$ is orthogonality to $Q$. Indeed
\begin{align*}
(Q,S_{1,0}\rho_1+\Lambda\rho_1-2T_{2,0})&=
(QS_{1,0},\rho_1)-(\Lambda Q,\rho_1)-2(Q,T_{2,0})\\
&=(QS_{1,0},\rho_1)-(S_{1,0},L_{-}\rho_1)+(S_{1,0},S_{1,0})\\
&=-(S_{1,0},L_{+}\rho_1)+(S_{1,0},S_{1,0})=0,
\end{align*}
using that $(S_{1,0},S_{1,0})=-2(T_{2,0},Q)$ (this relation can be obtained in Lemma \ref{lemma-3app} or see \cite{GL2022CPDE}), and the definition of $\rho_1$. Moreover, we clearly see that $S_{0,1,j}\rho_1-\partial_{y_j}\rho_1-T_{1,1,j}\bot Q$, since $S_{0,1,j}$ and $T_{1,1,j}$ are the antisymmetry functions, whereas $\rho_1$ and $Q$ are radial symmetry functions. Hence $\rho_2$ is well-defined.

We refer to Appendix \ref{section-app-mod-1} for some standard arguments, which show that the orthogonality condition \eqref{mod-orthogonality-condition} imply that the modulation parameters $\{a(t),b(t),\lambda(t),\alpha(t),\gamma(t)\}$ are uniquely determined, provided that $\epsilon=\epsilon_1+i\epsilon_2\in H^{s}(\mathbb{R}^2)$ is sufficiently small. Moreover, it follows from the standard arguments that $\{a(t),b(t),\lambda(t),\alpha(t),\gamma(t)\}$ are $C^1$-functions.

Inserting the decomposition \eqref{mod-decomposition} into \eqref{equ-1-hf-2}, we can obtain the following system
\begin{align}\label{mod-system-1}
&\left(a_s+\frac{1}{2}a^2\right)\partial_aQ_{1\mathcal{P}}+\sum_{j=1}^2((b_j)s+ab_j)\partial_{b_j}Q_{1\mathcal{P}}
+\partial_s\epsilon_1-M_{-}(\epsilon)+a\Lambda\epsilon_1-b\cdot\nabla\epsilon_1\notag\\
=&\left(\frac{\lambda_s}{\lambda}+a\right)(\Lambda Q_{1\mathcal{P}}+\Lambda\epsilon_1)
+\left(\frac{\alpha_s}{\lambda}-b\right)\cdot(\nabla Q_{1\mathcal{P}}+\nabla\epsilon_1)\notag\\
&+\tilde{\gamma}_s(Q_{2\mathcal{P}}+\epsilon_2)+\Im(\Phi_{\mathcal{P}})
-R_2(\epsilon),\\\label{mod-system-2}
&\left(a_s+\frac{1}{2}a^2\right)\partial_aQ_{2\mathcal{P}}+\sum_{j=1}^2((b_j)s+ab_j)\partial_{b_j}Q_{2\mathcal{P}}
+\partial_s\epsilon_2+M_{+}(\epsilon)+a\Lambda\epsilon_2-b\cdot\nabla\epsilon_2\notag\\
=&\left(\frac{\lambda_s}{\lambda}+a\right)(\Lambda Q_{2\mathcal{P}}+\Lambda\epsilon_2)
+\left(\frac{\alpha_s}{\lambda}-b\right)\cdot(\nabla Q_{2\mathcal{P}}+\nabla\epsilon_2)\notag\\
&-\tilde{\gamma}_s(Q_{1\mathcal{P}}-\epsilon_1)-\Re(\Phi_{\mathcal{P}})+R_1(\epsilon),
\end{align}
where $j=1,2$.
Here $\Phi_{\mathcal{P}}$ denotes the error term from Lemma \ref{lemma-3app}, and $M=(M_{+},M_{-})$ are the small deformations of the linearized operator $L=(L_{+},L_{-})$ given by
\begin{align}\label{mod-define-M1}
M_{+}(\epsilon)=&D\epsilon_1+\epsilon_1
-\frac{3}{2}|Q_{\mathcal{P}}|\epsilon_1
-\frac{1}{2}|Q_{\mathcal{P}}|^{-1}(Q_{1\mathcal{P}}^2-Q_{2\mathcal{P}}^2)\epsilon_1-|Q_{\mathcal{P}}|^{-1}Q_{1\mathcal{P}}Q_{2\mathcal{P}}\epsilon_2,\\\label{mod-define-M2}
M_{-}(\epsilon)=&D\epsilon_2+\epsilon_2
-\frac{3}{2}|Q_{\mathcal{P}}|\epsilon_2
-\frac{1}{2}|Q_{\mathcal{P}}|^{-1}(Q_{1\mathcal{P}}^2-Q_{2\mathcal{P}}^2)\epsilon_2-|Q_{\mathcal{P}}|^{-1}Q_{1\mathcal{P}}Q_{2\mathcal{P}}\epsilon_1.
\end{align}
And $R_1(\epsilon)$, $R_2(\epsilon)$ are the high order terms about $\epsilon$.
\begin{align*}
R_1(\epsilon)=&\frac{5}{4}|Q_{\mathcal{P}}|^{-1}Q_{1\mathcal{P}}|\epsilon|^2
+\frac{3}{8}|Q_{\mathcal{P}}|^{-1}\left(Q_{1\mathcal{P}}(\epsilon_1^2-\epsilon_2^2)
+2Q_{2\mathcal{P}}\epsilon_1\epsilon_2\right)\\
&+\frac{1}{8}|Q_{\mathcal{P}}|^{-3}\left((Q_{1\mathcal{P}}^3-3Q_{1\mathcal{P}}Q_{2\mathcal{P}}^2)(\epsilon_1-\epsilon_2)
+2(3Q_{1\mathcal{P}}^2Q_{2\mathcal{P}}-Q_{2\mathcal{P}}^3)\epsilon_1\epsilon_2\right)+\mathcal{O}(\epsilon^3),\\
R_2(\epsilon)=&\frac{5}{4}|Q_{\mathcal{P}}|^{-1}Q_{2\mathcal{P}}|\epsilon|^2
+\frac{3}{8}|Q_{\mathcal{P}}|^{-1}\left(2Q_{1\mathcal{P}}\epsilon_1\epsilon_2
+Q_{2\mathcal{P}}(\epsilon_1^2-\epsilon_2^2)\right)\\
&+\frac{1}{8}|Q_{\mathcal{P}}|^{-3}\left(2(Q_{1\mathcal{P}}^3-3Q_{1\mathcal{P}}Q_{2\mathcal{P}}^2)
\epsilon_2\epsilon_2+(3Q_{1\mathcal{P}}^2Q_{2\mathcal{P}}-Q_{2\mathcal{P}}^3)
(\epsilon_1^2-\epsilon_2^2)\right)+\mathcal{O}(\epsilon^3).
\end{align*}

 We have the following mixed energy and momentum type bound.
\begin{lemma}\label{lemma-mod-1}
For $t\in[t_0,t_1]$ with $t_1<0$, it holds that
\begin{align}\notag
a^2+|b|+\|\epsilon\|_{H^{1/2}}^2\lesssim\lambda(|E_0|+|P_0|)
+\mathcal{O}(\lambda^2+a^4+|b|^2+|b||\mathcal{P}|^2).
\end{align}
Here $E_0=E(u_0)$ and $P_0=P(u_0)$ denote the conserved energy and linear momentum of $u=u(t,x)$, respectively.
\end{lemma}
\begin{proof}
By similar arguments as \cite{GL2022CPDE} and combining the Lemma \ref{lemma-app-coercivity-estimate}, one can obtain the energy estimates. Now we give the momentum estimate. Let $v=Q_{\mathcal{P}}+\epsilon$,
we derive the bound for the boost parameter $b$. Here we observe that
\begin{align}\notag
P(v)=\lambda P(u_0),
\end{align}
by scaling and using the conservation of the linear momentum $P(u(t))=P(u_0)$. Hence, by expansion of $P(u)$, Lemma \ref{lemma-3app-2} and using the orthogonality condition \eqref{mod-orthogonality-condition}, we obtain
\begin{align*}
\lambda P_0=P(v)&=P(Q_{\mathcal{P}})+2\Re(\epsilon,-i\nabla Q_{\mathcal{P}})+\Re(\epsilon,-i\nabla\epsilon)\\
&=p_1 b+\mathcal{O}(a^4+|b|^2+|b||\mathcal{P}|^2+\|\epsilon\|_{H^{1/2}}^2),
\end{align*}
with the constant $p_1=2\int_{\mathbb{R}^2}(L_{-}S_{0,1}\cdot S_{0,1})>0$. Now we complete the proof of this Lemma.
\end{proof}

\subsection{Modulation Estimates}
We continue with estimating the modulation parameters. To this end, we define the vector-valued function
\begin{align*}
\mathbf{Mod}(t):=\left(a_s+\frac{1}{2}a^2,\tilde{\gamma}_s,\frac{\lambda_s}{\lambda}+a,
\frac{\alpha_s}{\lambda}-b,b_s+ab\right).
\end{align*}
We have the following result.
\begin{lemma}\label{lemma-mod-2}
For $t\in[t_0,t_1]$ with $t_1<0$, we have the bound
\begin{align*}
|\mathbf{Mod}(t)|\lesssim\lambda^2+a^4+|b|^2+|b||\mathcal{P}|^2+|\mathcal{P}|^2\|\epsilon\|_2+
\|\epsilon\|_2^2+\|\epsilon\|_{H^{1/2}}^3.
\end{align*}
Furthermore, we have the improved bound
\begin{align}\notag
\left|\frac{\lambda_s}{\lambda}+a\right|\lesssim a^5 +|b||\mathcal{P}|^2+|\mathcal{P}|^2\|\epsilon\|_2
+\|\epsilon\|_2^2+\|\epsilon\|_{H^{1/2}}^3.
\end{align}
\end{lemma}
\begin{proof}
We divide the proof into the following four steps, where we also make use of the estimates \eqref{app-B-1}-\eqref{app-B-5}, which are shown in Lemma \ref{lemma-app-mod-estimate}. Now, we recall that
\begin{align*}\notag
\Lambda Q_{1\mathcal{P}}=\Lambda Q+\mathcal{O}(|\mathcal{P}|^2),\ \Lambda Q_{2\mathcal{P}}=a\Lambda S_{1,0}+b\cdot\Lambda S_{0,1}+\mathcal{O}(|\mathcal{P}|^2),\\
\partial_aQ_{1\mathcal{P}}=2aT_{2,0}+b\cdot T_{1,1}+\mathcal{O}(|\mathcal{P}|^2),\ \partial_aQ_{2\mathcal{P}}=S_{1,0}+\mathcal{O}(|\mathcal{P}|^2),\\
\partial_{j} Q_{1\mathcal{P}}=\partial_{j} Q+\mathcal{O}(|\mathcal{P}|^2),\
\partial_{j} Q_{2\mathcal{P}}=a\partial_{j} S_{1,0}+\partial_{j}( b\cdot S_{0,1})+\mathcal{O}(|\mathcal{P}|^2),\\
\partial_{b_j}Q_{1\mathcal{P}}=aT_{1,1,j}+2b_jT_{0,2,j}+\mathcal{O}(|\mathcal{P}|^2),\ \partial_{b_j}Q_{2\mathcal{P}}=S_{0,1,j}+\mathcal{O}(|\mathcal{P}|^2),
\end{align*}
where $j=1,2$.

{\bf Step 1: Law for $b$}. We multiply both sides of the equation \eqref{mod-system-1} and \eqref{mod-system-2} by $-\partial_{j} Q_{2\mathcal{P}}$ and $\partial_{j} Q_{1\mathcal{P}}$, respectively. Adding this and using \eqref{app-B-4} yields, after some calculation (also using the condition \eqref{mod-orthogonality-condition}),
\begin{align*}
&\left(a_s+\frac{1}{2}a^2\right)[(\partial_aQ_{1\mathcal{P}},-\partial_{j} Q_{2\mathcal{P}})+
(\partial_aQ_{2\mathcal{P}},\partial_{j} Q_{1\mathcal{P}})]
+\sum_{j=1}^2((b_j)_s+ab_j)[(\partial_{b_j}Q_{1\mathcal{P}},-\partial_{j} Q_{2\mathcal{P}})\\
&+(\partial_{b_j}Q_{2\mathcal{P}},\partial_{j} Q_{1\mathcal{P}})]
+[(\partial_s\epsilon_1,-\partial_{j} Q_{2\mathcal{P}})+(\partial_s\epsilon_2,\partial_{j} Q_{1\mathcal{P}})]\\
=&\left(\frac{\lambda_s}{\lambda}+a\right)[(\Lambda Q_{1\mathcal{P}}+\Lambda\epsilon_1,-\partial_{j} Q_{2\mathcal{P}})+
(\Lambda Q_{2\mathcal{P}}+\Lambda\epsilon_2,\partial_{j} Q_{1\mathcal{P}})]\\
&+\left[\left(\left(\frac{\alpha_s}{\lambda}-b\right)\cdot(\nabla Q_{1\mathcal{P}}+\nabla\epsilon_1),-\partial_{j} Q_{2\mathcal{P}}\right)+\left(\left(\frac{\alpha_s}{\lambda}-b\right)\cdot(\nabla Q_{2\mathcal{P}}+\nabla\epsilon_2),\partial_{j} Q_{1\mathcal{P}}\right)\right]\\
&+\tilde{\gamma}_s[(Q_{2\mathcal{P}}+\epsilon_2,-\partial_{j} Q_{2\mathcal{P}})
-(Q_{1\mathcal{P}}+\epsilon_1,\partial_{j} Q_{1\mathcal{P}})]
+(R_2(\epsilon),\partial_{j} Q_{2\mathcal{P}})+(R_1(\epsilon),\partial_{j} Q_{1\mathcal{P}})\\
&-(\Im(\Phi_{\mathcal{P}}), \partial_{j} Q_{2\mathcal{P}})
+(\Re(\Phi_{\mathcal{P}}),\partial_{j} Q_{1\mathcal{P}})
+\mathcal{O}(\mathcal{P}^2\|\epsilon\|_2).
\end{align*}
Hence, we have
\begin{align*}
&\left(a_s+\frac{1}{2}a^2\right)[(S_{1,0},\partial_{j} Q)+\mathcal{O}(\mathcal{P}^2)]
+[(b_s+ab)\cdot(S_{0,1},\partial_{j} Q)+\mathcal{O}(\mathcal{P}^2)]\\
=&(R_2(\epsilon),\partial_{j} Q_{2\mathcal{P}})+(R_1(\epsilon),\partial_{j} Q_{1\mathcal{P}})
+\mathcal{O}\left((|\mathcal{P}|^2
+|\mathbf{mod}(t)|)\|\epsilon\|_2+a^4+|b|^2+|b||\mathcal{P}|^2\right).
\end{align*}
Therefore, we deduce that
\begin{align*}
&(b_s+ab)\left[-\frac{1}{2}p_1+\mathcal{O}(\mathcal{P}^2)\right]\\
=&(R_2(\epsilon),\nabla Q_{2\mathcal{P}})+(R_1(\epsilon),\nabla Q_{1\mathcal{P}})
+\mathcal{O}\left((\mathcal{P}^2
+|\mathbf{mod}(t)|)\|\epsilon\|_2+a^4+|b|^2+|b||\mathcal{P}|^2\right).
\end{align*}

{\bf Step 2: Law for $\alpha$}. We multiply both sides of the equation \eqref{mod-system-1} and \eqref{mod-system-2} by $-\partial_{b_j} Q_{2\mathcal{P}}$ and $\partial_{b_j} Q_{1\mathcal{P}}$, where $j=1,2$ respectively. Adding this and using \eqref{app-B-5} yields, after some calculation (also using the condition \eqref{mod-orthogonality-condition}), 
\begin{align*}
&\left(a_s+\frac{1}{2}a^2\right)[(\partial_aQ_{1\mathcal{P}},-\partial_{b_j} Q_{2\mathcal{P}})+
(\partial_aQ_{2\mathcal{P}},\partial_{b_j} Q_{1\mathcal{P}})]
+((b_j)_s+ab_j)[(\partial_{b_j}Q_{1\mathcal{P}},-\partial_{b_j} Q_{2\mathcal{P}})\\
&+(\partial_{b_j}Q_{2\mathcal{P}},\partial_{b_j} Q_{1\mathcal{P}})]
+[(\partial_s\epsilon_1,-\partial_b Q_{2\mathcal{P}})+(\partial_s\epsilon_2,\partial_{b_j} Q_{1\mathcal{P}})]\\
=&\left(\frac{\lambda_s}{\lambda}+a\right)[(\Lambda Q_{1\mathcal{P}}+\Lambda\epsilon_1,-\partial_{b_j} Q_{2\mathcal{P}})+
(\Lambda Q_{2\mathcal{P}}+\Lambda\epsilon_2,\partial_{b_j} Q_{1\mathcal{P}})]\\
&+\left(\frac{\alpha_s}{\lambda}-b\right)\cdot[(\nabla Q_{1\mathcal{P}}+\nabla\epsilon_1,-\partial_{b_j} Q_{2\mathcal{P}})+(\nabla Q_{2\mathcal{P}}+\nabla\epsilon_2,\partial_{b_j} Q_{1\mathcal{P}})]
+\tilde{\gamma}_s[(Q_{2\mathcal{P}}+\epsilon_2,-\partial_{b_j} Q_{2\mathcal{P}})\\
&-(Q_{1\mathcal{P}}+\epsilon_1,\partial_{b_j} Q_{1\mathcal{P}})]
+(R_2(\epsilon),\partial_{b_j} Q_{2\mathcal{P}})+(R_1(\epsilon),\partial_{b_j} Q_{1\mathcal{P}})
-(\Im(\Phi_{\mathcal{P}}), \partial_{b_j} Q_{2\mathcal{P}})\\
&+(\Re(\Phi_{\mathcal{P}}),\partial_{b_j} Q_{1\mathcal{P}})
+\mathcal{O}(\mathcal{P}^2\|\epsilon\|_2).
\end{align*}
Hence, we have
\begin{align*}
&\left(a_s+\frac{1}{2}a^2\right)\mathcal{O}(\mathcal{P})\\
=&\left(\frac{\lambda_s}{\lambda}+a\right)[-(\Lambda Q,S_{0,1,j})+\mathcal{O}(\mathcal{P}^2)]
+\left(\frac{\alpha_s}{\lambda}-b\right)\cdot[(\nabla Q,S_{0,1,j})+\mathcal{O}(|\mathcal{P}|^2)]\\
&+(R_2(\epsilon),\partial_b Q_{2\mathcal{P}})+(R_1(\epsilon),\partial_b Q_{1\mathcal{P}})+\mathcal{O}(\mathcal{P}^2\|\epsilon\|_2).
\end{align*}
Therefore, we deduce that
\begin{align*}
&\left(a_s+\frac{1}{2}a^2\right)\mathcal{O}(\mathcal{P})
+\left(\frac{\alpha_s}{\lambda}-b\right)[p_1 +\mathcal{O}(\mathcal{P}^2)]\\
=&(R_2(\epsilon),\partial_{b_j} Q_{2\mathcal{P}})+(R_1(\epsilon),\partial_{b_j} Q_{1\mathcal{P}})
+\mathcal{O}\left((|\mathcal{P}|^2
+|\mathbf{mod}(t)|)\|\epsilon\|_2+a^4+|b|^2+|b||\mathcal{P}|^2\right),
\end{align*}
where we used $(\Lambda Q,S_{0,1})=0$.

{\bf Step~3:} By using the estimates \eqref{app-B-1}, \eqref{app-B-2} and \eqref{app-B-3} and the similar methods in \cite[Lemma 4.2]{GL2022CPDE}, we can obtain the laws for $a$, $\lambda$ and $\gamma$. Here we omit it.

{\bf Step 4: Conclusion}.  We collect the previous equation and estimate the nonlinear terms in $\epsilon$ by Sobolev inequalities. This gives us
\begin{align*}
(A+B)\mathbf{Mod}(t)=&\mathcal{O}\big((\mathcal{P}^2+|\mathbf{Mod}(t)|)\|\epsilon_2\|
+\|\epsilon\|_2^2+\|\epsilon\|_{H^{1/2}}^3\\
&+|\|u\|_2^2-\|Q\|_2^2|+a^4+|b|^2+|b||\mathcal{P}|^2\big).
\end{align*}
Here $A=O(1)$ is invertible $7\times7$-matrix, and $B=\mathcal{O}(\mathcal{P})$ is some $7\times7$-matrix that is polynomial in $\mathcal{P}=(a,b)$. For $|\mathcal{P}|\ll1$, we can thus invert $A+B$ by Taylor expansion and derive the estimate for $\mathbf{Mod}(t)$ stated in this Lemma.
\end{proof}

\subsection{Refined Energy bounds}
In this section, we establish a refined energy estimate, which will be a key ingredient in the compactness argument to construct the ground state mass blowup solutions.

Let $u=u(t,x)$ be a solution \eqref{equ-1-hf-2} on the time interval $[t_0,0)$ and suppose that $\tilde{Q}$ is an approximate solution to \eqref{equ-1-hf-2} such that
\begin{equation*}
i\tilde{Q}_t-D\tilde{Q}+|\tilde{Q}|\tilde{Q}=\psi,
\end{equation*}
with the priori bounds
\begin{align*}
\|\tilde{Q}\|_2\lesssim 1,\ \|D^{\frac{1}{2}}\tilde{Q}\|_2\lesssim \lambda^{-\frac{1}{2}},\ \|\nabla \tilde{Q}\|_2\lesssim \lambda^{-1}.
\end{align*}
We decompose $u=\tilde{Q}+\tilde{\epsilon}$, and hence $\tilde{\epsilon}$ satisfies
\begin{equation*}
i\tilde{\epsilon}_t-D\tilde{\epsilon}+(|u|u-|\tilde{Q}|\tilde{Q})=-\psi,
\end{equation*}
where we assume the priori estimate
\begin{align*}
\|D^{\frac{1}{2}}\tilde{\epsilon}\|_2\lesssim \lambda^{\frac{1}{2}},\ \|\tilde{\epsilon}\|_2\lesssim \lambda,
\end{align*}
as well as
\begin{align*}
|\lambda_t+a|\lesssim\lambda^2,\ a\lesssim\lambda^{\frac{1}{2}},\ |a_t|\lesssim1,\ |\alpha_t|\lesssim\lambda.
\end{align*}

Next, Let $\phi:\mathbb{R}^2\rightarrow\mathbb{R}$ be a smooth and radial function with the following properties
\begin{align*}
\phi^\prime(x)=\begin{cases}x\ \ &\text{for}\ \ 0\leq x\leq1,\\
3-e^{-|x|}\ &\text{for}\ x\geq2,
\end{cases}
\end{align*}
and the convexity condition
\begin{align*}
\phi''(x)\geq0\ \text{for}\ x\geq0.
\end{align*}
Furthermore, we denote
\begin{align}\notag
F(u)=\frac{1}{3}|u|^{3},\ f(u)=|u|u,\ F'(u)\cdot h=\Re(f(u)\bar{h}).
\end{align}
Let $A>0$ be a large constant and define the quantity
\begin{align*}
J_A(u):=&\frac{1}{2}\int|D^{\frac{1}{2}}\tilde{\epsilon}|^2
+\frac{1}{2}\int\frac{|\tilde{\epsilon}|^2}{\lambda}
-\int[F(u)-F(\tilde{Q})-F'(\tilde{Q})\cdot\tilde{\epsilon}]\\
&+\frac{a}{2}\Im\left(\int A\nabla\phi\left(\frac{x-\alpha}{A\lambda}\right)\cdot(\nabla\tilde{\epsilon})\bar{\tilde{\epsilon}}\right).
\end{align*}
Our strategy will be to use the preceding functional to bootstrap control over $\|\tilde{\epsilon}\|_{H^{s}}$.
\begin{lemma}\label{lemma-energy-estimate}
(Localized energy estimate) Let $J_A$ be as above. Then we have

\begin{align*}
\frac{dJ_A}{dt}
=&\Im\left(\psi,D\tilde{\epsilon}+\frac{1}{\lambda}\tilde{\epsilon}-f'(\tilde{Q})\tilde{\epsilon}\right)
-\frac{1}{\lambda}(\tilde{\epsilon},f'(\tilde{Q})\tilde{\epsilon})
-\Re\left(\partial_t\tilde{Q},{(f(u)-f(\tilde{Q})-f'(\tilde{Q})\cdot\tilde{\epsilon})}\right)\\
&+\frac{a}{2\lambda}\int\frac{|\tilde{\epsilon}|^2}{\lambda}
-\frac{2a}{\lambda}\int_0^{+\infty}\sqrt{s}
\int_{\mathbb{R}^2}\Delta\phi\left(\frac{x-\alpha}{A\lambda}\right)|\nabla\tilde{\epsilon}_s|^2dxds\\
&+\frac{a}{2A^2\lambda^3}\int_0^{+\infty}\sqrt{s}\int_{\mathbb{R}^2}
\Delta^2\phi\left(\frac{x-\alpha}{A\lambda}\right)|\tilde{\epsilon}_s|^2dxds\\
&+\Im\left(\int\left[iaA\nabla\phi\left(\frac{x-\alpha}{A\lambda}\right)\cdot\nabla\psi
+i\frac{a}{2\lambda}\Delta\psi\left(\frac{x-\alpha}{A\lambda}\right)\psi\right]\bar{\tilde{\epsilon}}\right)\\
&+a\Re\left(\int A\nabla\phi\left(\frac{x-\alpha}{A\lambda}\right)
\left(\frac{3}{4}|\tilde{Q}|^{-1}|\tilde{\epsilon}|^2w\tilde{Q}
+\frac{1}{4}|\tilde{Q}|^{-1}\tilde{\epsilon}^2\bar{\tilde{Q}}\right)\cdot\overline{\nabla \tilde{Q}}\right)\\
&+\mathcal{O}(\lambda^2\|\psi\|_2+\|\tilde{\epsilon}\|_{H^{1/2}}^2
+\lambda^{\frac{1}{2}}\|\tilde{\epsilon}\|_{H^{1/2}}^2),
\end{align*}
where $\tilde{\epsilon}_s:=\sqrt{\frac{2}{\pi}}\frac{1}{-\Delta+s}\tilde{\epsilon}$ with $s>0$.
\end{lemma}
\begin{proof} {The proof is similar to the proof of Lemma 5.1 in \cite{GL2022CPDE}. A small modification due to the presence of the translation $\alpha$ is needed. For this we omit the details of the proof.} 
\end{proof}

\subsection{Backwards Propagation of Smallness}
We now apply the energy estimate of the previous subsection in order to establish a bootstrap argument that will be needed in the construction of the ground state mass blowup solution.

Let $u=u(t,x)$ be a solution to \eqref{equ-1-hf-2} defined in $[t_0,0)$. Assume that $t_0<t_1<0$ and suppose that $u$ admits on $[t_0,t_1]$ a geometrical decomposition of the form
\begin{align*}
u(t,x)=\frac{1}{\lambda(t)}[Q_{\mathcal{P}(t)}+\epsilon]
\left(s,\frac{x-\alpha(t)}{\lambda(t)}\right)
e^{i\gamma(t)},
\end{align*}
where $\epsilon=\epsilon_1+i\epsilon_2$ satisfies the orthogonality condition \eqref{mod-orthogonality-condition} and $a^2+|b|+\|\epsilon\|_{H^{1/2}}^2\ll1$ holds. We set
\begin{align*}
\tilde{\epsilon}(t,x)=\frac{1}{\lambda(t)}\epsilon
\left(s,\frac{x-\alpha(t)}{\lambda(t)}\right)
e^{i\gamma(t)}.
\end{align*}
Suppose that the energy satisfies $E_0=E(u_0)>0$ and define the constant
\begin{align}\label{back-define-1}
A_0=\sqrt{\frac{e_1}{E_0}},
\end{align}
with the constant $e_1=\frac{1}{2}(L_{-}S_{1,0},S_{1,0})>0$. Moreover, Let $P_0=P(u_0)$ be the linear momentum and define the vector
\begin{align}\label{back-define-2}
B_0=\frac{P_0}{p_1},
\end{align}
where $p_1=2\int_{\mathbb{R}^2}L_{-}S_{0,1}\cdot S_{0,1}>0$ is a constant.

Now we claim that the following backwards propagation estimate holds.
\begin{lemma}(Backwards propagation of smallness)\label{lemma-back}
Assume that, for some $t_1<0$ sufficiently close to $0$, we have the bounds
\begin{align*}
&\left|\|u\|_2^2-\|Q\|_2^2\right|\lesssim\lambda^2(t_1),\\
&\|D^{\frac{1}{2}}\tilde{\epsilon}(t_1)\|_2^2+\frac{\|\tilde{\epsilon}\|_2^2}{\lambda(t_1)}
\lesssim\lambda^2(t_1),\\
&\left|\lambda(t_1)-\frac{t_1^2}{4A_0^2}\right|\lesssim\lambda^{3}(t_1),\
\left|\frac{a(t_1)}{\lambda^{\frac{1}{2}}(t_1)}\right|\lesssim\lambda(t_1),\
\left|\frac{b(t_1)}{\lambda(t_1)}-B_0\right|\lesssim\lambda(t_1),
\end{align*}
where $A_0$ and $B_0$ are defined in \eqref{back-define-1} and \eqref{back-define-2}, respectively. Then there exists a time $t_0<t_1$ depending on $A_0$ and $B_0$ such that $\forall t\in[t_0,t_1]$, it holds that
\begin{align*}
&\|D^{\frac{1}{2}}\tilde{\epsilon}(t)\|_2^2+\frac{\|\tilde{\epsilon}\|_2^2}{\lambda(t)}\lesssim
\|D^{\frac{1}{2}}\tilde{\epsilon}(t_1)\|_2^2+\frac{\|\tilde{\epsilon}\|_2^2}{\lambda(t_1)}\lesssim\lambda^2(t),\\
&\left|\lambda(t)-\frac{t^2}{4A_0^2}\right|\lesssim\lambda^{3}(t),\
\left|\frac{a(t)}{\lambda^{\frac{1}{2}}(t)}-\frac{1}{A_0}\right|\lesssim\lambda(t),\
\left|\frac{b(t)}{\lambda(t)}-B_0\right|\lesssim\lambda(t).
\end{align*}
\end{lemma}
\begin{proof}
By using the Lemma \ref{lemma-mod-2} and \ref{lemma-energy-estimate} and the similar argument as \cite{GL2022CPDE}, we can obtain the estimates about $\tilde{\epsilon}$ in $H^\frac{1}{2}$, $\lambda(t)$ and $a(t)$. Here we only give the estimate for $b$. In fact, by following the calculations in the proof of Lemma \ref{lemma-mod-1} for the linear momentum $P(u_0)$ and recalling that $a^2+|b|\sim\lambda$, we deduce that
\begin{align}\notag
bp_1=\lambda P_0+\mathcal{O}(\lambda^2),
\end{align}
where $p_1=2\int_{\mathbb{R}^2}L_{-}S_{0,1}\cdot S_{0,1}$ is a positive constant. Here we also used the fact that $\|\epsilon\|_{H^{\frac{1}{2}}}^2\lesssim\lambda^2$. Recalling the definition of $B_0=\frac{P_0}{p_1}$ see \eqref{back-define-2}, we thus obtain
\begin{align*}
\left|\frac{b(t)}{\lambda(t)}-B_0\right|\lesssim\lambda(t).
\end{align*}
This complete the proof of Lemma \ref{lemma-back}.
\end{proof}

\section{Existence of ground state mass blowup solutions}
In this section, we prove the following result.
\begin{theorem}
Let $\gamma_0\in\mathbb{R}$, $P_0\in\mathbb{R}^2$, $x_0\in\mathbb{R}^2$ and $E_0>0$ be given. Then there exist a time $t_0<0$ and a solution $u\in C([t_0,0); H^{s}(\mathbb{R}^2))$, $\delta\in\left(\frac{3}{4},1\right)$ of \eqref{equ-1-hf-2} such that $u$ blowup  at time $T=0$ with
\begin{align}\notag
E(u)=E_0,\ P(u)=P_0,\ \text{and}\ \|u\|_2^2=\|Q\|_2^2.
\end{align}
Furthermore, we have $\|D^{\frac{1}{2}}u\|_2\sim t^{-1}$ as $t\rightarrow0^{-}$, and $u$ is of the form
\begin{align}\notag
u(t,x)=\frac{1}{\lambda(t)}[Q_{\mathcal{P}(t)}+\epsilon]
\left(t,\frac{x-\alpha}{\lambda}\right)e^{i\gamma(t)}=\tilde{Q}+\tilde{\epsilon},
\end{align}
where $\mathcal{P}(t)=(a(t),b(t))$, and $\epsilon$ satisfies the orthogonality condition \eqref{mod-orthogonality-condition}. Finally, the following estimate hold:
\begin{align*}
&\|\tilde{\epsilon}\|_2\lesssim\lambda,\ \|\tilde{\epsilon}\|_{H^{1/2}}\lesssim\lambda^{\frac{1}{2}},\\
&\lambda(t)-\frac{t^2}{4A_0^2}=\mathcal{O}(\lambda^2),\ \frac{a}{\lambda^{\frac{1}{2}}}(t)-\frac{1}{A_0}=\mathcal{O}(\lambda),\
\frac{b}{\lambda}(t)-B_0=\mathcal{O}(\lambda),\\
&\gamma(t)=-\frac{4A_0^2}{t}+\gamma_0+\mathcal{O}(\lambda^{\frac{1}{2}}),\
\alpha(t)=x_0+\mathcal{O}(\lambda^{\frac{3}{2}}),
\end{align*}
 for $t\in[t_0,0)$ and $t$ sufficiently close to $0$.
Here $A_0>0$ and $B_0\in\mathbb{R}^2$  is a  constant and the vector defined in \eqref{back-define-1} and \eqref{back-define-2}, respectively.
\end{theorem}
\begin{proof}\textbf{Step 1.} Backwards uniform bounds.

Let $t_n\rightarrow0^{-}$ be a sequence of negative times and let $u_n$ be the solution to \eqref{equ-1-hf-2} with initial data at $t=t_n$ given by
\begin{align*}
u_n(t_n,x)=\frac{1}{\lambda_n(t_n)}Q_{\mathcal{P}_n(t_n)}
\left(\frac{x-\alpha_n(t_n)}{\lambda_n(t_n)}\right)e^{i\gamma_n(t_n)},
\end{align*}
where the sequence $\mathcal{P}_n(t_n)=(a_n(t_n),b_n(t_n))$ and $\{\lambda_n(t_n),\alpha_n(t_n)\}$ are given by
\begin{align}\notag
a_n(t_n)&=-\frac{t_n}{2A_0},\ \lambda_n(t_n)=\frac{t_n^2}{4A_0^2},\
\gamma_n(t_n)=\gamma_0-\frac{4A_0^2}{t_n},\\\label{exist:1}
b_n(t_n)&=\frac{B_0t_n^2}{2A_0},\ \alpha_n(t_n)=x_0.
\end{align}
By Lemma \ref{lemma-3app-2}, we have
\begin{align*}
\int|u_n(t_n)|^2=\int|Q|^2+\mathcal{O}(t_n^4),
\end{align*}
and $\tilde{\epsilon}(t_n)=0$ by construction. Thus $u_n$ satisfies the assumptions of Lemma \ref{lemma-back}. Hence we can find a backwards time $t_0$ independent of $n$ such that for al $t\in[t_0,t_n)$ we have the geometric decomposition
\begin{align*}
u_n(t,x)=\frac{1}{\lambda_n(t)}Q_{\mathcal{P}_n(t)}
\left(\frac{x-\alpha_n(t)}{\lambda_n(t)}\right)+\tilde{\epsilon}_n(t,x),
\end{align*}
with the uniform bounds given by
\begin{align*}
&\|D^{\frac{1}{2}}\tilde{\epsilon}_n(t)\|_2^2
+\frac{\|\tilde{\epsilon}_n(t)\|_2^2}{\lambda_n(t)}\lesssim\lambda_n(t),\\
&\left|\lambda_n(t)-\frac{t^2}{4A_0^2}\right|\lesssim K\lambda_n^{2}(t),\
\left|\frac{a_n(t)}{\lambda_n^{\frac{1}{2}}(t)}-\frac{1}{A_0}\right|\lesssim K\lambda_n(t),\
\left|\frac{b_n(t)}{\lambda_n(t)}-B_0\right|\lesssim K\lambda_n(t),
\end{align*}
and
\begin{align}\label{Hs:estimate}
\|\tilde{\epsilon}_n(t)\|_{\dot{H}^{\frac{1}{2}+\theta}}\lesssim\lambda_n^{\frac{1}{2}-\theta},~~\text{where}~~~\theta\in\left(\frac{1}{4},\frac{1}{2}\right),
\end{align}
which we prove in step 2.

The proof of similar to the proof in \cite[Theorem 7.1]{GL2022CPDE}. A small modification due to the presence of the translation parameter $\alpha$ is needed.
For this, we notice that 
\[P(u(t))=\frac{b}{\lambda}p_1+o(1)\to P_0~~\text{as}~~t\to 0^-,\]
by the choice of $B_0$ and $\lambda_n(t_n)$. By linear momentum conservation, this shows that 
\[P(u(t))=P_0.\]
Finally, we recall the rough bound $\left|\frac{(\alpha_n)_s}{\lambda_n}+b_n\right|\lesssim\lambda_n$. Integrating this and using the bounds for $b_n$ and $\lambda_n$, we deduce that 
\[\left|\frac{d}{dt}(\alpha_n-x_0)\right|=\left|\frac{(\alpha_n)_s}{\lambda_n}\right|\lesssim\lambda_n+|b_n|\lesssim\lambda_n.\]
Integrating this bound and using \eqref{exist:1}, we find that 
\[\alpha_n(t)=x_0+\mathcal{O}\left(\lambda^{\frac{3}{2}}\right),\]
which shows that the claim for $\alpha(t)$ holds.

\textbf{Step 2.} $H^{\frac{1}{2}+\theta}$ bound.


It remains to prove the $H^{1/2+\theta}$ bound \eqref{Hs:estimate}. Our point of departure is again the identity
\begin{align*}
    i\partial_t\tilde{\epsilon}_n=D\tilde{\epsilon}_n-\psi_n-F_n,
\end{align*}
where
\begin{align*}
    F_n=|\tilde{Q}_n+\tilde{\epsilon}_n|(\tilde{Q}_n+\tilde{\epsilon}_n)-|\tilde{Q}_n|\tilde{Q}_n
\end{align*}
We plan to obtain a $H^{1/2+\theta}$-bound on $\tilde{\epsilon}_n$, taking advantage of the a priori bounds at time $t_n \sim \lambda_n^{1/2}$ and those assumed for $t\in[t_0,t_n]$.
We make  partition of the interval $[t_0,t_n]$ into
$$ t_0 = s_0 < s_1 < \cdots < s_N=t_n,  \ s_j-s_{j-1}= h, \ j = 1, \cdots , N, $$
where
$$  h \sim \lambda_n^2,~ N \sim (t_n-t_0)/h \sim (\lambda_n)^{-\frac{3}{2}} .$$
We can obtain estimates in each interval $\Delta_j=[s_{j-1},s_j]$.  To obtain the desired result, for any time interval $I$ we use the space $X_{I}$ with norm
\begin{align*}
    \|u\|_{X_{I}}=\|u(t,x)\|_{L^{\infty}_{I} H^s(\mathbb{R}^2)}+\|u(t,x)\|_{L^q_IL^\infty(\mathbb{R}^2)},
\end{align*}
where $\frac{1}{q}\geq 1-s$, $q>4$ and $s<1$.

Now we estimate the nonlinear term
\begin{align*}
   F_n=&|\tilde{Q}_n+\tilde{\epsilon}_n|(\tilde{Q}_n+\tilde{\epsilon}_n)-|\tilde{Q}_n|\tilde{Q}_n\\
   =&|\tilde{Q}_n+\tilde{\epsilon}_n|\tilde{\epsilon}_n+(|\tilde{Q}_n+\tilde{\epsilon}_n|-|\tilde{Q}_n|)\tilde{Q}_n \\
   \lesssim&|\tilde{Q}_n|\tilde{\epsilon}_n+|\tilde{\epsilon}_n|\tilde{\epsilon}_n+(|\tilde{Q}_n+\tilde{\epsilon}_n|-|\tilde{Q}_n|)\tilde{Q}_n.
\end{align*}
To this end we can use the following. 
\begin{lemma}\label{lemma:tec}
    There exists $\delta>0$ so that for any $s \in  [0,1]$ there exists a constant $C=C(\delta,s)>0$ so that for any $u \in H^s \cap L^\infty$ such that
    $$ \left\|\frac{u}{Q_n} \right\|_{L^\infty} \leq \delta,$$
        we have 
        $$ \| |Q_n+u|-|Q_n| \|_{H^s} \leq C \|u\|_{H^s}.$$
\end{lemma}
\begin{proof}
    It is sufficient to make the series expansion
    $$|Q_n+u|-|Q_n| = |Q_n| \left( \sum_{k+\ell \geq 1} c_{k,\ell} \left(\frac{u}{Q_n}\right)^k \left(\frac{\overline{u}}{\overline{Q_n}}\right)^\ell\right)$$
    combined with the fractional Leibniz rule of Lemma \ref{lemma:flr} and
    \begin{equation}\label{eq.mhs1}
        \left\||Q_n| \frac{u}{Q_n} \right\|_{H^s} \lesssim \|u\|_{H^s}.
    \end{equation}
        The last estimate can be verified for $s=1$ and then by complex interpolation we have \eqref{eq.mhs1}.
\end{proof}
Hence, by Lemma \ref{lemma:flr} and \ref{lemma:fcr}, we get 
\begin{align*}
  \|F_n\|_{H^s}=&\||\tilde{Q}_n+\tilde{\epsilon}_n|\tilde{\epsilon}_n+(|\tilde{Q}_n+\tilde{\epsilon}_n|-|\tilde{Q}_n|)\tilde{Q}_n\|_{H^s}\\
  \lesssim&\|\tilde{Q}_n\|_{H^s}\|\tilde{\epsilon}_n\|_{L^\infty}+\|\tilde{Q}_n\|_{L^\infty}\|\tilde{\epsilon}_n\|_{H^s}+\|(|\tilde{Q}_n+\tilde{\epsilon}_n|-|\tilde{Q}_n|)\|_{H^s} \|\tilde{Q}_n\|_{L^\infty}\\
  &+\|(|\tilde{Q}_n+\tilde{\epsilon}_n|-|\tilde{Q}_n|)\|_{L^\infty} \|\tilde{Q}_n\|_{H^s} \\
\lesssim&\|\tilde{Q}_n\|_{H^s}\|\tilde{\epsilon}_n\|_{L^\infty}+\|\tilde{Q}_n\|_{L^\infty}\|\tilde{\epsilon}_n\|_{H^s}+\|\tilde{\epsilon}_n\|_{L^\infty}\|\tilde{\epsilon}_n\|_{H^s},
\end{align*}
where in the last step we used the Lemma \ref{lemma:tec}.

Now defining the nonlinear mapping
\begin{align*}
    \Psi(u)(t):=U(t_0)u(t_0,x)-i\int_{t_0}^t U(t-s)[F_n(s)+\psi_n(s)]ds.
\end{align*}
Then 
\begin{align*}
    \|\Psi(u)(t)\|_{H^s}\lesssim&\|u(t_0)\|_{H^s}+\int_I\left(\|F_n\|_{H^s}+\|\psi_n\|_{H^s}\right)dt\\
    \lesssim&\|u(t_0)\|_{H^s}+\int_I[\|\tilde{Q}_n\|_{L^\infty}\|\tilde{\epsilon}_n\|_{H^s}+\|\tilde{Q}_n\|_{H^s}\|\tilde{\epsilon}\|_{L^\infty}+\|\tilde{\epsilon}_n\|_{L^\infty}\|\tilde{\epsilon}_n\|_{H^s}+\|\psi_n\|_{H^s}]dt\\
    \lesssim&\|u(t_0)\|_{H^s}+{|I|}^{1-\frac{1}{q}}\Big(\|\tilde{Q}_n\|_{L^qL^\infty}\|\tilde{\epsilon}_n\|_{L^\infty H^s}+\|\tilde{Q}_n\|_{L^\infty H^s}\|\tilde{\epsilon}\|_{L^qL^\infty}\\
    &+\|\tilde{\epsilon}_n\|_{L^qL^\infty}\|\tilde{\epsilon}_n\|_{L^\infty H^s}\Big)+\int_I\|\psi_n\|_{H^s}dt\\
    \lesssim&\|u(t_0)\|_{H^s}+{|I|}^{1-\frac{1}{q}}\Big(|I|^{\frac{1}{q}}\lambda_n^{-1}\|\tilde{\epsilon}_n\|_{L^\infty H^s}+\lambda^{-s}\|\tilde{\epsilon}\|_{L^qL^\infty}\\
    &+\|\tilde{\epsilon}_n\|_{L^qL^\infty}\|\tilde{\epsilon}_n\|_{L^\infty H^s}\Big)+\int_I\|\psi_n\|_{H^s}dt,
\end{align*}
where $I=[t_0,t]$ is the time interval and the $L^q_IL^\infty$-norm 
\begin{align*}
    &\|\Psi(u)\|_{L^q_IL^\infty}\\
    \lesssim&\|U(t)u(t_0)\|_{L^q_IL^\infty}+\|F_n+\psi_n\|_{L^1_I\dot{H}^{1-\frac{1}{q}}}\\
    \lesssim&\|u(t_0)\|_{\dot{H}^{1-\frac{1}{q}}}+\int_I\|F_n+\psi_n\|_{H^s}\\
    \lesssim&\|u(t_0)\|_{\dot{H}^s}+\int_I[\|\tilde{Q}_n\|_{L^\infty}\|\tilde{\epsilon}_n\|_{H^s}+\|\tilde{Q}_n\|_{H^s}\|\tilde{\epsilon}\|_{L^\infty}+\|\tilde{\epsilon}_n\|_{L^\infty}\|\tilde{\epsilon}_n\|_{H^s}+\|\psi_n\|_{H^s}]\\
    \lesssim&\|u(t_0)\|_{\dot{H}^s}+\|\tilde{\epsilon}_n\|_{L^\infty H^s}\int_I\|\tilde{Q}_n\|_{L^\infty}+\|\tilde{Q}_n\|_{L^\infty H^s}\int_I\|\tilde{\epsilon}_n\|_{L^\infty}\\
    &+\|\tilde{\epsilon}_n\|_{L^\infty H^s}\int_I\|\tilde{\epsilon}_n\|_{L^\infty}+\|\psi_n\|_{L^1\dot{H}^s}\\
    \lesssim&\|u(t_0)\|_{\dot{H}^s}+|I|\lambda_n^{-1}\|\tilde{\epsilon}_n\|_{L^\infty H^s}+|I|^{1-\frac{1}{q}}\lambda_n^{-s}\|\tilde{\epsilon}_n\|_{L^qL^\infty}\\
    &+|I|^{1-\frac{1}{q}}\|\tilde{\epsilon}_n\|_{L^qL^\infty}\|\tilde{\epsilon}_n\|_{L^\infty H^s}+|I|\|\psi_n\|_{L^\infty\dot{H}^s}.
\end{align*}
Here $s\geq1-\frac{1}{q}$, $H^s\subset \dot{H}^{1-\frac{1}{q}}$ and we used Strichartz estimates (see Lemma \ref{lemma:stri}) and the estimates $\|\tilde{Q}_n\|_{L^\infty}\lesssim\lambda_n^{-1}$ and $\|\tilde{Q}_n\|_{H^s}\lesssim\lambda_n^{-s}$ again.


Now we consider the interval $\Delta_N=[s_{N-1},s_N]$,
\begin{align*}
    \|\tilde{\epsilon}_n\|_{X_{\Delta_N}}=&\|\tilde{\epsilon}_n\|_{L^{\infty}_{X_{\Delta_N}} H^s(\mathbb{R}^2)}+\|\tilde{\epsilon}_n\|_{L^q_{X_{\Delta_N}}L^\infty(\mathbb{R}^2)}\\
    \lesssim&|h|^{1-\frac{1}{q}}\Big(|h|^{\frac{1}{q}}\lambda_n^{-1}\|\tilde{\epsilon}_n\|_{L^\infty_{\Delta_N} H^s}+\lambda^{-s}\|\tilde{\epsilon}\|_{L^q_{\Delta_N}L^\infty}+\|\tilde{\epsilon}_n\|_{L^q_{\Delta_N}L^\infty}\|\tilde{\epsilon}_n\|_{L^\infty_{\Delta_N} H^s}\Big)\\
    &+|h|\|\psi_n\|_{L^\infty_{\Delta_N} H^s}+|h|\lambda_n^{-1}\|\tilde{\epsilon}_n\|_{L^\infty_{\Delta_N} H^s}+|h|^{1-\frac{1}{q}}\lambda_n^{-s}\|\tilde{\epsilon}_n\|_{L^q_{\Delta_N}L^\infty}\\
    &+|h|^{1-\frac{1}{q}}\|\tilde{\epsilon}_n\|_{L^q_{\Delta_N}L^\infty}\|\tilde{\epsilon}_n\|_{L^\infty_{\Delta_N} H^s}+|h|\|\psi_n\|_{L^\infty_{\Delta_N}\dot{H}^s}\\
    \lesssim&\lambda_n^{(1-\frac{1}{q})}\Big(\lambda_n^{\frac{1}{q}}\lambda_n^{-1}\|\tilde{\epsilon}_n\|_{L^\infty_{\Delta_N} H^s}+\lambda^{-s}\|\tilde{\epsilon}\|_{L^q_{\Delta_N}L^\infty}+\|\tilde{\epsilon}_n\|_{L^q_{\Delta_N}L^\infty}\|\tilde{\epsilon}_n\|_{L^\infty_{\Delta_N} H^s}\Big)\\
    &+\lambda_n^2\lambda_n^{-1}\|\tilde{\epsilon}_n\|_{L^\infty_{\Delta_N} H^s}+(\lambda_n^2)^{1-\frac{1}{q}}\lambda_n^{-s}\|\tilde{\epsilon}_n\|_{L^q_{\Delta_N}L^\infty}\\
    &+\lambda_n^{1-\frac{1}{q}}\|\tilde{\epsilon}_n\|_{L^q_{\Delta_N}L^\infty}\|\tilde{\epsilon}_n\|_{L^\infty_{\Delta_N} H^s}+\lambda_n^2\|\psi_n\|_{X_{\Delta_N}}\\
    \lesssim&\lambda_n\|\tilde{\epsilon}_n\|_{X_{\Delta_N}}+\lambda_n^{2-\frac{2}{q}-s}\|\tilde{\epsilon}_n\|_{X_{\Delta_N}}+\lambda_n^{2-\frac{2}{q}}\|\tilde{\epsilon}_n\|_{X_{\Delta_N}}^2+\lambda_n^2\|\psi_n\|_{X_{\Delta_N}}.
\end{align*}
Here we used $\tilde{\epsilon}(t_n)=0$ and $h\sim\lambda_n^2$. Since $q>4$, $1>s\geq1-\frac{1}{q}$  and $\|\tilde{\epsilon}_n\|_{X_{\Delta_N}}$ is small, the above estimate implies
\begin{align*}
     \|\tilde{\epsilon}_n\|_{X_{\Delta_N}}\lesssim \lambda_n^2\|\psi_n\|_{L^\infty_{\Delta_N}{H}^s}.
\end{align*}
The term $\|\psi_n\|_{L^\infty_{\Delta_N}{H}^s}.$ can be estimated by the following estimates 
\[\|\nabla^k\psi_n\|_{L^2} \lesssim \lambda_n^{1-k},~~k=0,1.
\]
This estimates can be obtained by \cite[Lemma 6.1. Step 4]{GL2022CPDE} and a small modification due to the presence of the translation parameter $\alpha$. 
So that
\begin{align*}
   \|\psi_n\|_{L^{\infty}_{I} H^s(\mathbb{R}^2)}
   \lesssim \|\psi_n\|_{L^2}^{1-s}\|\psi_n\|_{H^1}^s\lesssim\lambda_n^{1-s}.
\end{align*}
Hence, we deduce that 
\begin{align*}
     \|\tilde{\epsilon}_n\|_{X_{\Delta_N}}\lesssim \lambda_n^{3-s}.
\end{align*}
For the other intervals $\Delta_{j}$, $j=1,\cdots,N-1$,  we have
\begin{align*}
 \|\tilde{\epsilon}_n\|_{X_{\Delta_{j}}}\lesssim& \|\tilde{\epsilon}_n(s_{j})\|_{H^s}+\lambda_n^2\|\tilde{\epsilon}_n\|_{X_{\Delta_j}}+\lambda_n^{2(1-\frac{1}{q})-s}\|\tilde{\epsilon}_n\|_{X_{\Delta_j}}\\
 &+\lambda_n^{2(1-\frac{1}{q})}\|\tilde{\epsilon}_n\|_{X_{\Delta_j}}^2+\lambda_n^2\|\psi_n\|_{L^\infty_{\Delta_j}H^s}.
\end{align*}
By the similar argument as before, we deduce
\begin{align*}
\|\tilde{\epsilon}_n\|_{X_{\Delta_{j}}}\lesssim& \|\tilde{\epsilon}_n(s_{j})\|_{H^s}+\lambda_n^{3-s} \lesssim \|\tilde{\epsilon}_n\|_{X_{\Delta_{j+1}}}+\lambda_n^{3-s} .
\end{align*}
and inductively we find
\[\|\tilde{\epsilon}_n\|_{X_{\Delta_{j}}} \lesssim  (N-j)\|\tilde{\epsilon}_n\|_{X_{\Delta_{N}}}+\lambda_n^{3-s} \lesssim (N-j+1)\lambda_n^{3-s}  .\]
Therefore,  we have
\begin{align*}
\|\tilde{\epsilon}_n\|_{X_{[t_0,t_n]}}=&\|\tilde{\epsilon}_n\|_{L_{t_0,t_n]}^\infty H^s}+\|\tilde{\epsilon}_n\|_{L^q_{[t_0,t_n]}L^\infty}\\
=&\sup_{1\leq j\leq N}\|\tilde{\epsilon}_n\|_{L_{\Delta_{j}}^\infty H^s}+|N|^{\frac{1}{q}}\sup_{1\leq j\leq N}\|\tilde{\epsilon}_n\|_{L^q_{\Delta_{j}}L^\infty}\\
\lesssim&N\lambda_n^{3-s}+N^{\frac{1}{q}}N\lambda_n^{3-s}\lesssim\lambda_n^{\frac{3}{2}\left(1-\frac{1}{q}\right)-s}.
\end{align*}
Since $q>4$ and $1-\frac{1}{q}\leq s<1$, then
\[\frac{3}{2}\left(1-\frac{1}{q}\right)-s>\frac{3}{2}\left(1-\frac{1}{4}\right)-s=\frac{9}{8}-s>0.\]
This means
\begin{align*}
\|D^s \tilde{\epsilon}_n\|_{L^2}\lesssim \lambda_n^{9/8 - s} \lesssim  \lambda_n^{1 - s},~~\text{for}~~s\in\left(\frac{3}{4},1\right).
\end{align*}
With $s=1/2+\theta,$ then we get \eqref{Hs:estimate}.

Now the proof of this Theorem is  complete.
\end{proof}

\appendix

\section{Appendix A}
\subsection{Uniqueness of modulation parameters}\label{section-app-mod-1}
First, we show that the parameters $\{a,b,\lambda,\alpha,\gamma\}$ are uniquely determined if $\epsilon=\epsilon_1+i\epsilon_2\in H^{1/2}(\mathbb{R}^2)$ is sufficiently small and satisfies the orthogonality conditions \eqref{mod-orthogonality-condition}. Indeed, this follows from an implicit function argument, which we detail here.

For $\delta>0$, let $W_{\delta}=\{w\in H^{1/2}(\mathbb{R}^2):\|w-Q\|_{H^{1/2}}<\delta\}$. Consider approximate blowup profiles $Q_{\mathcal{P}}$ with $|\mathcal{P}|=|(a,b)|<\eta$, where $\eta>0$ is a small constant. For $w\in W_{\delta}$, $\lambda_1>0$, $y_1\in\mathbb{R}^2$, $\gamma\in\mathbb{R}$ and $|\mathcal{P}|<\eta$, we define
\begin{align}\notag
\epsilon_{\lambda_1,y_1,\gamma_1,a,b}(y)
=e^{i\gamma_1}\lambda_1w(\lambda_1y-y_0)-Q_{\mathcal{P}}.
\end{align}
Consider the map $\mathbf{\sigma}=(\sigma^1,\sigma^2,\sigma^3,\sigma^4,\sigma^5,\sigma^6,\sigma^7)$ define by
\begin{align*}
\sigma^1&=((\epsilon_{\lambda_1,y_0,\gamma_1,a,b})_2,\Lambda Q_{1\mathcal{P}})-
((\epsilon_{\lambda_1,y_1,\gamma_1,a,b})_1,\Lambda Q_{2\mathcal{P}}),\\
\sigma^2&=((\epsilon_{\lambda_1,y_1,\gamma_1,a,b})_2,\partial_aQ_{1\mathcal{P}})-
((\epsilon_{\lambda_1,y_1,\gamma_1,a,b})_1,\partial_aQ_{2\mathcal{P}}),\\
\sigma^3&=((\epsilon_{\lambda_1,y_1,\gamma_1,a,b})_1,\rho_2)-
((\epsilon_{\lambda_1,y_1,\gamma_1,a,b})_2,\rho_1),\\
\sigma^4&=((\epsilon_{\lambda_1,y_1,\gamma_1,a,b})_2,\partial_{1} Q_{1\mathcal{P}})-
((\epsilon_{\lambda_1,y_1,\gamma_1,a,b})_1,\partial_{1} Q_{2\mathcal{P}}),\\
\sigma^5&=((\epsilon_{\lambda_1,y_1,\gamma_1,a,b})_2,\partial_{2} Q_{1\mathcal{P}})-
((\epsilon_{\lambda_1,y_1,\gamma_1,a,b})_1,\partial_{2} Q_{2\mathcal{P}}),\\
\sigma^6&=((\epsilon_{\lambda_1,y_1,\gamma_1,a,b})_2,\partial_{b_1} Q_{1\mathcal{P}})-
((\epsilon_{\lambda_1,y_1,\gamma_1,a,b})_1,\partial_{b_1} Q_{2\mathcal{P}}),\\
\sigma^7&=((\epsilon_{\lambda_1,y_1,\gamma_1,a,b})_2,\partial_{b_2} Q_{1\mathcal{P}})-
((\epsilon_{\lambda_1,y_1,\gamma_1,a,b})_1,\partial_{b_2} Q_{2\mathcal{P}}).
\end{align*}
Recall that $\rho=\rho_1+i\rho_2$ was defined in \eqref{mod-definition-rho}. Taking the partial derivatives at $(\lambda_1,y_1,\gamma_1,a,b_1,b_2)=(1,0,0,0,0,0)$ yields that
\begin{align*}
\frac{\partial\epsilon_{\lambda_1,y_1,\gamma_1,a,b}}{\partial\lambda_1}=\Lambda w, \frac{\partial\epsilon_{\lambda_1,y_1,\gamma_1,a,b}}{\partial y_{1,j}}=-\partial_{j} w,
\frac{\partial\epsilon_{\lambda_1,y_1,\gamma_1,a,b}}{\partial\gamma_1}=iw,\\
\frac{\partial\epsilon_{\lambda_1,y_1,\gamma_1,a,b}}{\partial a}=-\partial_aQ_{\mathcal{P}}|_{\mathcal{P}=(0,0)}=-iS_{1,0},\\
\frac{\partial\epsilon_{\lambda_1,y_1,\gamma_1,a,b}}{\partial b_j}=-\partial_{b_j}Q_{\mathcal{P}}|_{\mathcal{P}=(0,0)}=-iS_{0,1,j},
\end{align*}
where we recall that $L_{-}S_{1,0}=\Lambda Q$ and $L_{-}S_{0,1}=-\nabla Q$. Note that $S_{1,0}$ is an radial function, whereas $S_{0,1}$ is antisymmetry. At $(\lambda_1,y_1,\gamma_1,a,b_1,b_2,w)=(1,0,0,0,0,0,0,Q)$, the Jacobian of the map $\sigma$ is hence given by
\begin{align*}
\frac{\partial \sigma^1}{\partial\lambda_1}=0,\ \frac{\partial\sigma^1}{\partial y_{1,1}}=0,\ \frac{\partial\sigma^1}{\partial y_{1,2}}=0,\ \frac{\partial \sigma^1}{\partial\gamma_1}=0,\ \frac{\partial \sigma^1}{\partial a}=-(S_{1,0},L_{-}S_{1,0}),\ \frac{\partial \sigma^1}{\partial b_1}=0,\ \frac{\partial \sigma^1}{\partial b_2}=0,\\
\frac{\partial \sigma^2}{\partial\lambda_1}=-(S_{1,0},L_{-}S_{1,0}),\ \frac{\partial\sigma^2}{\partial y_{1,1}}=0,\ \frac{\partial\sigma^2}{\partial y_{1,2}}=0,\ \frac{\partial \sigma^2}{\partial\gamma_1}=0,\ \frac{\partial \sigma^2}{\partial a}=0,\ \frac{\partial \sigma^2}{\partial b_1}=0,\ \frac{\partial \sigma^2}{\partial b_2}=0,\\
\frac{\partial \sigma^3}{\partial\lambda_1}=0,\ \frac{\partial\sigma^3}{\partial y_{1,1}}=0,\ \frac{\partial\sigma^3}{\partial y_{1,2}}=0,\ \frac{\partial \sigma^3}{\partial\gamma_1}=-(Q,\rho_1),\ \frac{\partial \sigma^3}{\partial a}=0,\ \frac{\partial \sigma^3}{\partial b_1}=0,\ \frac{\partial \sigma^3}{\partial b_2}=0,\\
\frac{\partial \sigma^4}{\partial\lambda_1}=0,\ \frac{\partial\sigma^4}{\partial y_{1,1}}=0,\ \frac{\partial\sigma^4}{\partial y_{1,2}}=0,\ \frac{\partial \sigma^4}{\partial\gamma_1}=0,\ \frac{\partial \sigma^4}{\partial a}=0,\ \frac{\partial \sigma^4}{\partial b_1}=-(L_{-}S_{0,1,1},S_{0,1,1}),\ \frac{\partial \sigma^4}{\partial b_2}=0,\\
\frac{\partial \sigma^5}{\partial\lambda_1}=0,\ \frac{\partial\sigma^5}{\partial y_{1,1}}=0,\ \frac{\partial\sigma^5}{\partial y_{1,2}}=0,\ \frac{\partial \sigma^5}{\partial\gamma_1}=0,\ \frac{\partial \sigma^5}{\partial a}=0,\ \frac{\partial \sigma^5}{\partial b_1}=0,\ \frac{\partial \sigma^5}{\partial b_2}=-(L_{-}S_{0,1,2},S_{0,1,2}),\\
\frac{\partial \sigma^6}{\partial\lambda_1}=0,\ \frac{\partial\sigma^6}{\partial y_{1,1}}=(L_{-}S_{0,1,1},S_{0,1,1}),\ \frac{\partial\sigma^6}{\partial y_{1,2}}=0,\  \frac{\partial \sigma^6}{\partial\gamma_1}=0,\ \frac{\partial \sigma^6}{\partial a}=0,\ \frac{\partial \sigma^6}{\partial b_1}=0,\ \frac{\partial \sigma^6}{\partial b_2}=0,\\
\frac{\partial \sigma^7}{\partial\lambda_1}=0,\ \frac{\partial\sigma^7}{\partial y_{1,1}}=0,\ \frac{\partial\sigma^7}{\partial y_{1,2}}=(L_{-}S_{0,1,2},S_{0,1,2}),\  \frac{\partial \sigma^7}{\partial\gamma_1}=0,\ \frac{\partial \sigma^7}{\partial a}=0,\ \frac{\partial \sigma^7}{\partial b_1}=0,\ \frac{\partial \sigma^7}{\partial b_2}=0.
\end{align*}
Note that we used here that $Q$ and $S_{1,0}$ are the radial symmetry functions, whereas $S_{0,1}$ is antisymmetry, for example $(Q,S_{0,1})=0$. Moreover, we note
\begin{align}\notag
-(Q,\rho_1)=(L_{+}\Lambda Q,\rho_1)=-(\Lambda Q,L_{+}\rho_1)=-(\Lambda Q,S_{1,0})=-(L_{-}S_{1,0},S_{1,0}).
\end{align}
Since $(L_{-}S_{1,0},S_{1,0})>0$ and $(L_{-}S_{0,1,j},S_{0,1,j})>0$, hence the determinant of the functional matrix is nonzero. By the implicit function theorem, we obtain existence and uniqueness for $(\lambda_1,y_1,\gamma_1,a,b_1,b_2,w)$ in some neighborhood around $(1,0,0,0,0,0,0,Q)$.

\subsection{Estimates for the modulation equations}
To conclude this section, we collect some estimates needed in the discussion of the modulation equations in section \ref{section-mod-estimate}.
\begin{lemma}\label{lemma-app-mod-estimate}
The following estimates hold
\begin{align}\label{app-B-1}
&(M_{-}(\epsilon)-a\Lambda\epsilon_1+b\cdot\nabla\epsilon_1,\Lambda Q_{2\mathcal{P}})+(M_{+}(\epsilon)+a\Lambda\epsilon_2-b\cdot\nabla\epsilon_2,\Lambda Q_{1\mathcal{P}})\notag\\
=&-\Re(\epsilon,Q_{\mathcal{P}})
+\mathcal{O}(\mathcal{P}^2\|\epsilon\|_2),\\\label{app-B-2}
&(M_{-}(\epsilon)-a\Lambda\epsilon_1+b\cdot\nabla\epsilon_1,\partial_a Q_{2\mathcal{P}})+(M_{+}(\epsilon)+a\Lambda\epsilon_2-b\cdot\nabla\epsilon_2,
\partial_a Q_{1\mathcal{P}})
=\mathcal{O}(\mathcal{P}^2\|\epsilon\|_2),\\\label{app-B-3}
&(M_{-}(\epsilon)-a\Lambda\epsilon_1+b\cdot\nabla\epsilon_1,\rho_2)
+(M_{+}(\epsilon)+a\Lambda\epsilon_2-b\cdot\nabla\epsilon_2,\rho_1 )
=\mathcal{O}(\mathcal{P}^2\|\epsilon\|_2),\\\label{app-B-4}
&(M_{-}(\epsilon)-a\Lambda\epsilon_1+b\cdot\nabla\epsilon_1,\partial_j Q_{2\mathcal{P}})+(M_{+}(\epsilon)+a\Lambda\epsilon_2-b\cdot\nabla\epsilon_2,\partial_{j} Q_{1\mathcal{P}})
=\mathcal{O}(\mathcal{P}^2\|\epsilon\|_2),\\\label{app-B-5}
&(M_{-}(\epsilon)-a\Lambda\epsilon_1+b\cdot\nabla\epsilon_1,\partial_{b_j} Q_{2\mathcal{P}})+(M_{+}(\epsilon)+a\Lambda\epsilon_2-b\cdot\nabla\epsilon_2,\partial_{b_j} Q_{1\mathcal{P}})
=\mathcal{O}(\mathcal{P}^2\|\epsilon\|_2).
\end{align}
\end{lemma}
\begin{proof}
First, we recall \eqref{mod-define-M1} and \eqref{mod-define-M2} that
\begin{align*}
M_{+}(\epsilon)=&L_{+}\epsilon_1-|Q_{\mathcal{P}}|^{-1}
Q_{1\mathcal{P}}Q_{2\mathcal{P}}\epsilon_2+\mathcal{O}(\mathcal{P}\epsilon),\\
M_{-}(\epsilon)=&L_{-}\epsilon_2-|Q_{\mathcal{P}}|^{-1}
Q_{1\mathcal{P}}Q_{2\mathcal{P}}\epsilon_1+\mathcal{O}(\mathcal{P}\epsilon).
\end{align*}
We have notice the identity
\begin{align}\label{app-B-6}
L_{-}\Lambda S_{1,0}=-S_{1,0}+2(\Lambda Q)QS_{1,0}+\Lambda Q+\Lambda^2Q.
\end{align}
To see this relation, we recall that $L_{-}S_{1,0}=\Lambda Q$ and hence
\begin{align*}
L_{-}\Lambda S_{1,0}&=[L_{-},\Lambda]S_{1,0}+\Lambda L_{-}S_{1,0}=DS_{1,0}+(y\cdot\nabla Q)S_{1,0}+\Lambda^2Q\\
&=-S_{1,0}+|Q|S_{1,0}+\Lambda Q+(y\cdot\nabla Q)S_{1,0}+\Lambda^2Q\\
&=-S_{1,0}+\Lambda Q+(\Lambda Q)S_{1,0}+\Lambda^2Q.
\end{align*}
Hence, \eqref{app-B-6} holds. Similarly, we deduce from $L_{-}S_{0,1}=-\nabla Q$ that
\begin{align}\label{app-B-7}
L_{-}\Lambda S_{0,1}=-S_{0,1}-\nabla Q+(\Lambda Q)S_{0,1}-\Lambda\nabla Q.
\end{align}
Next, we recall that
\begin{align}\notag
\Lambda Q_{1\mathcal{P}}=\Lambda Q+\mathcal{O}(\mathcal{P}^2),\ \Lambda Q_{2\mathcal{P}}=a\Lambda S_{1,0}+b\cdot\Lambda S_{0,1}+\mathcal{O}(\mathcal{P}^2).
\end{align}
Combining \eqref{app-B-6} and \eqref{app-B-7} with the fact and using that $L_{+}\Lambda Q=-Q$, we find that
\begin{align*}
&\text{left-hand side of \eqref{app-B-1}}\\
=&a(\epsilon_2,L_{-}\Lambda S_{1,0})+(\epsilon_2, b\cdot L_{-}\Lambda S_{0,1})-
|Q_{\mathcal{P}}|^{-1}
(Q_{1\mathcal{P}}Q_{2\mathcal{P}}\epsilon_1,a\Lambda S_{1,0}+b\cdot\Lambda S_{0,1})\\
&+(-a\Lambda\epsilon_1+b\cdot\Lambda\nabla\epsilon_1,a\Lambda S_{1,0}+b\cdot\Lambda S_{0,1})+(\epsilon_1,L_{+}\Lambda Q)+(a\Lambda\epsilon_2-b\cdot\nabla\epsilon_2,\Lambda Q)\\&-|Q_{\mathcal{P}}|^{-1}
(Q_{1\mathcal{P}}Q_{2\mathcal{P}}\epsilon_2,\Lambda Q)+\mathcal{O}(\mathcal{P}^2\|\epsilon\|_2)\\
=&-(\epsilon_1, Q)-a(\epsilon_2,S_{1,0})-(\epsilon_2,b\cdot S_{0,1})+a(\epsilon_2,\Lambda Q)
-(\epsilon_2,b\cdot\nabla Q)+\mathcal{O}(\mathcal{P}^2\|\epsilon\|_2)\\
=&-\Re(\epsilon, Q_{\mathcal{P}})+\mathcal{O}(\mathcal{P}^2\|\epsilon\|_2).
\end{align*}
Here we used that $a(\epsilon_2,\Lambda Q)=\mathcal{O}(\mathcal{P}^2\|\epsilon\|_2)$ and $(\epsilon_2,b\cdot\nabla Q)=\mathcal{O}(\mathcal{P}^2\|\epsilon\|_2)$, which follows from the orthogonality condition \eqref{mod-orthogonality-condition}.

$\mathbf{Estimate~\eqref{app-B-2}}$. From Lemma \ref{lemma-3app-2} we recall that
\begin{align*}
\partial_aQ_{1\mathcal{P}}=2aT_{2,0}+b\cdot T_{1,1}+\mathcal{O}(\mathcal{P}^2),\ \partial_aQ_{2\mathcal{P}}=S_{1,0}+\mathcal{O}(\mathcal{P}^2),
\end{align*}
where
\begin{align*}
L_{+}T_{2,0}=\frac{1}{2}S_{1,0}-\Lambda S_{1,0}+\frac{1}{2}|S_{1,0}|^2,\
L_{+}T_{1,1}=S_{0,1}-\Lambda S_{0,1}+\nabla S_{1,0}+S_{1,0}S_{0,1}.
\end{align*}
Using this fact, we have
\begin{align*}
&\text{left-hand side of \eqref{app-B-2}}\\
=&(\epsilon_2,L_{-}S_{1,0})-|Q_{\mathcal{P}}|^{-1}
(Q_{1\mathcal{P}}Q_{2\mathcal{P}}\epsilon_1,S_{1,0})+a(\epsilon_1,\Lambda S_{1,0})\\
&-(\epsilon_1,b\cdot\nabla S_{1,0})+2a(\epsilon_1,L_{+}T_{2,0})+(\epsilon_1,b\cdot L_{+}T_{1,1})+
\mathcal{O}(\mathcal{P}^2\|\epsilon\|_2)\\
=&(\epsilon_2,\Lambda Q)-|Q_{\mathcal{P}}|^{-1}
((aQS_{1,0}+b\cdot QS_{0,1})\epsilon_1,S_{1,0})+a(\epsilon_1,\Lambda S_{1,0})\\
&-(\epsilon_1,b\cdot\nabla S_{1,0})+2a\left(\epsilon_1,\frac{1}{2}S_{1,0}-\Lambda S_{1,0}+\frac{1}{2}|S_{1,0}|^2\right)\\
&+(\epsilon_1,b\cdot S_{0,1}-b\cdot\Lambda S_{0,1})+\nabla S_{1,0}+b\cdot S_{1,0}S_{0,1})+
\mathcal{O}(\mathcal{P}^2\|\epsilon\|_2)\\
=&(\epsilon_2,\Lambda Q)-a(\epsilon_1,\Lambda S_{1,0})-(\epsilon_1,b\cdot\Lambda S_{0,1})+(\epsilon_1,b\cdot S_{0,1})+\mathcal{O}(\mathcal{P}^2\|\epsilon\|_2)\\
=&(\epsilon_2,\Lambda Q_{1\mathcal{P}})-(\epsilon_1,\Lambda Q_{2\mathcal{P}})+
\mathcal{O}(\mathcal{P}^2\|\epsilon\|_2),
\end{align*}
where in the last step we also used that $(\epsilon_1,b\cdot S_{0,1})=\mathcal{O}(\mathcal{P}^2\|\epsilon\|_2)$, thanks to the orthogonality condition \eqref{mod-orthogonality-condition}.

$\mathbf{Estimate~\eqref{app-B-3}}$. Indeed, by the definition of $\rho=\rho_1+i\rho_2$, we have
\begin{align*}
&\text{left-hand side of \eqref{app-B-3}}\\
=&(\epsilon_2,L_{-}\rho_2)+(\epsilon_1,L_{+}\rho_1)-
|Q_{\mathcal{P}}|^{-1}((aQS_{1,0}+b\cdot QS_{0,1})\epsilon_2,\rho_1)\\
&-a(\epsilon_2,\Lambda\rho_1)+(\epsilon_2,b\cdot\nabla\rho_1)
+\mathcal{O}(\mathcal{P}^2\|\epsilon\|_2)\\
=&a(\epsilon_2,S_{1,0}\rho_1)+a(\epsilon_2,\Lambda\rho_1)
-2a(\epsilon_2,T_{2,0})+(\epsilon_2,b\cdot S_{0,1}\rho_1)\\
&-(\epsilon_2,b\cdot\nabla\rho_1)-(\epsilon_2,b\cdot T_{1,1})+(\epsilon_1,S_{1,0})-
|Q_{\mathcal{P}}|^{-1}((aQS_{1,0}+b\cdot QS_{0,1})\epsilon_2,\rho_1)\\
&-a(\epsilon_2,\Lambda\rho_1)+(\epsilon_2,b\cdot\nabla\rho_1)
+\mathcal{O}(\mathcal{P}^2\|\epsilon\|_2)\\
=&-2a(\epsilon_2,T_{2,0})-(\epsilon_2,b\cdot T_{1,1})+(\epsilon_1,S_{1,0})
+\mathcal{O}(\mathcal{P}^2\|\epsilon\|_2)\\
=&-(\epsilon_2,\partial_aQ_{1\mathcal{P}})+(\epsilon_1,\partial_aQ_{2\mathcal{P}})
+\mathcal{O}(\mathcal{P}^2\|\epsilon\|_2)\\
=&\mathcal{O}(\mathcal{P}^2\|\epsilon\|_2),
\end{align*}
where we use the orthogonality condition \eqref{mod-orthogonality-condition}.

$\mathbf{Estimate~\eqref{app-B-4}}$. First, we note that
\begin{align*}
\nabla Q_{1\mathcal{P}}=\nabla Q+\mathcal{O}(\mathcal{P}^2),\
\nabla Q_{2\mathcal{P}}=a\nabla S_{1,0}+\sum_{j=1}^2b_j\nabla S_{0,1,j}+\mathcal{O}(\mathcal{P}^2).
\end{align*}
Moreover, we have the relation
\begin{align*}
L_{+}\nabla Q&=0,\ L_{-}\nabla S_{1,0}=(\nabla Q)S_{1,0}+\nabla\Lambda Q,\\
L_{-}\nabla S_{0,1}&=(\nabla Q)\cdot S_{0,1}-\nabla \cdot\nabla Q.
\end{align*}
Indeed, since $[D,\nabla]=0$, we have
\begin{align*}
L_{+}\nabla Q=(D+1-2Q)\nabla Q=\nabla(DQ+Q-Q^{2})=0.
\end{align*}
Due to the fact that $[L_{-},\nabla]=[Q,\nabla]$, we obtain
\begin{align*}
L_{-}\nabla S_{1,0}=[L_{-},\nabla]S_{1,0}+\nabla L_{-}S_{1,0}=(\nabla Q)S_{1,0}+\nabla\Lambda Q,
\end{align*}
where we use $L_{-}S_{1,0}=\Lambda Q$. Similarly, we can obtain $L_{-}\nabla S_{0,1,j}=(\nabla Q) S_{0,1,j}-\nabla(\partial_{j} Q),\,j=1,2$. Thus, we deduce
\begin{align*}
&\text{left-hand side of \eqref{app-B-4}}\\
=&a(\epsilon_2,L_{-}\partial_j S_{1,0})+b_j(\epsilon_2,L_{-}\partial_j S_{0,1,j})+(\epsilon_1,L_{+}\partial_j Q)-a(\epsilon_2QS_{1,0},\partial_j Q)\\
&-b_j(\epsilon_2QS_{0,1,j},\partial_j Q)-a(\epsilon_2,\Lambda\partial_j Q)
+b_j(\epsilon_2,\partial_j(\partial_{j}Q))+\mathcal{O}(\mathcal{P}^2\|\epsilon\|_2)\\
=&a(\epsilon_2,\partial_j QS_{1,0}+\partial_j\Lambda Q)
+b_j(\epsilon_2,\partial_j QS_{0,1,j}-\partial_j(\partial_{j}Q))\\
&-a(\epsilon_2QS_{1,0},\partial_j Q)-b_j(\epsilon_2QS_{0,1,j},\partial_j Q)
-a(\epsilon_2,\Lambda\partial_j Q)\\
&+b_j(\epsilon_2,\partial_j(\partial_{j}Q))+\mathcal{O}(\mathcal{P}^2\|\epsilon\|_2)\\
=&a(\epsilon_2,[\partial_j,\Lambda]Q)+\mathcal{O}(\mathcal{P}^2\|\epsilon\|_2)\\
=&a(\epsilon_2,\partial_j Q)+\mathcal{O}(\mathcal{P}^2\|\epsilon\|_2)=
\mathcal{O}(\mathcal{P}^2\|\epsilon\|_2),
\end{align*}
since $a(\epsilon_2,\partial_jQ)=\mathcal{O}(\mathcal{P}^2\|\epsilon\|_2)$ due to the orthogonality condition \eqref{mod-orthogonality-condition}.

$\mathbf{Estimate~\eqref{app-B-5}}$. We note that
\begin{align*}
\partial_{b_j}Q_{1\mathcal{P}}=aT_{1,1,j}+2b_jT_{0,2,j}+\mathcal{O}(\mathcal{P}^2),\ \partial_{b_j}Q_{2\mathcal{P}}=S_{0,1,j}+\mathcal{O}(\mathcal{P}^2),
\end{align*}
where
\begin{align*}
L_{+}T_{0,2,j}=\partial_{x_j} S_{0,1,j}+\frac{1}{2}|S_{0,1,j}|^2.
\end{align*}
Using the above relations, we obtain that
\begin{align*}
&\text{left-hand side of \eqref{app-B-5}}\\
=&(\epsilon_2,L_{-}S_{0,1,j})-a(\epsilon_1QS_{1,0},S_{0,1,j})
-b_j(\epsilon_1S_{0,1,j}Q,S_{0,1})+a(\epsilon_1,\Lambda S_{0,1,j})\\
&-b_j(\epsilon_1,\partial_{j} S_{0,1,j})+a(\epsilon_1,L_{+}T_{1,1,j})+2b_j(\epsilon_2,L_{+}T_{0,2,j})
+\mathcal{O}(\mathcal{P}^2\|\epsilon\|_2)\\
=&-(\epsilon_2,\partial_{j} Q)-a(\epsilon_1QS_{1,0},S_{0,1,j})
-b_j(\epsilon_1S_{0,1,j}Q,S_{0,1,j})+a(\epsilon_1,\Lambda S_{0,1,j})\\
&-b_j(\epsilon_1,\partial_{j} S_{0,1,j})
+a\left(\epsilon_1,S_{0,1,j}-\Lambda S_{0,1,j}+\nabla S_{1,0}+S_{0,1,j}S_{1,0}\right)\\
&+2b_j\left(\epsilon_1,\partial_{j} S_{0,1,j}+\frac{1}{2}|S_{0,1,j}|^2\right)
+\mathcal{O}(\mathcal{P}^2\|\epsilon\|_2)\\
=&-(\epsilon_2,\partial_{j} Q)+a(\epsilon_1,\partial_{j} S_{1,0})+b_j(\epsilon_1,\partial_{j} S_{0,1,j})
+\mathcal{O}(\mathcal{P}^2\|\epsilon\|_2)\\
=&-(\epsilon_2,\partial_{j} Q_{1\mathcal{P}})+(\epsilon_1,\partial_{j} Q_{2\mathcal{P}})
+\mathcal{O}(\mathcal{P}^2\|\epsilon\|_2)\\
=&\mathcal{O}(\mathcal{P}^2\|\epsilon\|_2)
\end{align*}
where in the last step we use the orthogonality condition \eqref{mod-orthogonality-condition} and hence we proven  this Lemma.
\end{proof}

\subsection{Coercivity}

\begin{lemma}\label{lemma-app-coercivity-estimate}
There exist a constant $C_1>0$ such that for all $\epsilon=\epsilon_1+i\epsilon_2\in H^{1/2}(\mathbb{R}^2)$, we have the coercivity estimate
\begin{align}\notag
(L_{+}\epsilon_1,\epsilon_1)+(L_{-}\epsilon_2,\epsilon_2)\geq C_1\int|\epsilon|^2-\frac{1}{C_1}\left\{(\epsilon_1,Q)^2+(\epsilon_1,S_{1,0})^2
+|(\epsilon_1,S_{0,1})|^2+|(\epsilon_2,\rho_1)|^2\right\}.
\end{align}
Here $S_{1,0}$ and $S_{0,1}$ are the unique functions such that $L_{-}S_{1,0}=\Lambda Q$ with $(S_{1,0},Q)=0$ and $L_{-}S_{0,1}=-\nabla Q$ with $(S_{0,1},Q)=0$, respectively, and the function $\rho_1$ is defined in \eqref{mod-definition-rho}.
\end{lemma}




\renewcommand{\proofname}{\bf Proof.}

\noindent
{\bf Acknowledgments}

 V. Georgiev was partially supported by   Gruppo Nazionale per l'Analisi Matematica 2020, by the project PRIN  2020XB3EFL with the Italian Ministry of Universities and Research, by Institute of Mathematics and Informatics, Bulgarian Academy of Sciences, by Top Global University Project, Waseda University and the Project PRA 2022 85 of University of Pisa.
 Y.Li was supported by China Postdoctoral Science Foundation (No. 2021M701365) and the funding of innovating activities in Science and Technology of Hubei Province.


\vspace*{.5cm}


\bigskip

\begin{flushleft}
Vladimir Georgiev,\\
Dipartimento di Matematica, Universit\`{a} di Pisa, Largo B. Pontecorvo 5, 56127 Pisa, Italy\\
 Faculty of Science and Engineering, Waseda University, 3-4-1, Okubo, Shinjuku-ku, Tokyo 169-8555, Japan\\
 IMICBAS, Acad. Georgi Bonchev Str., Block 8, 1113 Sofia, Bulgaria\\
E-mail: georgiev@dm.unipi.it
\end{flushleft}

\begin{flushleft}
Yuan Li,\\
School of Mathematics and Statistics, Central China Normal University, Wuhan, PR China\\
E-mail: yli2021@ccnu.edu.cn
\end{flushleft}

\bigskip

\medskip

\end{document}